\documentclass[oneside,english]{amsart}
\usepackage[T1]{fontenc}
\usepackage[utf8]{inputenc}
\usepackage{amsthm}
\usepackage{amssymb}
\usepackage{esint}
\usepackage[all]{xy}
\usepackage{lmodern}
\usepackage{color}
\usepackage[bbgreekl]{mathbbol}

\makeatletter

\numberwithin{equation}{section}
\numberwithin{figure}{section}
\theoremstyle{plain}
\newtheorem{thm}{\protect\theoremname}[section]
  \theoremstyle{plain}
  \newtheorem{prop}[thm]{\protect\propositionname}
  \theoremstyle{definition}
  \newtheorem{defn}[thm]{\protect\definitionname}
  \theoremstyle{remark}
  \newtheorem{rem}[thm]{\protect\remarkname}
  \theoremstyle{plain}
  \newtheorem{lem}[thm]{\protect\lemmaname}
  \theoremstyle{definition}
  \newtheorem{example}[thm]{\protect\examplename}
  \newtheorem{cor}[thm]{\protect\corollaryname}
  \theoremstyle{plain}

\makeatother

\usepackage{babel}
  \providecommand{\definitionname}{Definition}
  \providecommand{\examplename}{Example}
  \providecommand{\lemmaname}{Lemma}
  \providecommand{\propositionname}{Proposition}
  \providecommand{\remarkname}{Remark}
  \providecommand{\corollaryname}{Corollary}
\providecommand{\theoremname}{Theorem}

\DeclareMathOperator{\Ch}{\mathcal{C}\mathit{h}}

\begin{document}

\title{Function spaces and classifying spaces of algebras over a prop}

\address{University of Copenhagen, Department of Mathematical Sciences,
Universitetsparken 5, 2100 København}

\email{yalinprop@gmail.com}

\author{Sinan Yalin}
\begin{abstract}
The goal of this paper is to prove that the classifying spaces of categories of algebras governed by a prop can be determined by using function spaces on the category of props.
We first consider a function space of props to define the moduli space of algebra structures over this prop on an object of the base category.
Then we mainly prove that this moduli space is the homotopy fiber of a forgetful map of classifying spaces, generalizing to the prop setting a theorem of Rezk.

The crux of our proof lies in the construction of certain universal diagrams in categories of algebras over a prop.
We introduce a general method to carry out such constructions in a functorial way.


\textit{AMS} : 18G55 ; 18D50 ; 18D10 ; 55U10.

\end{abstract}
\maketitle

\tableofcontents

\section*{Introduction}

Associative algebras, Lie algebras, Poisson algebras and their variants play a key role in algebra, topology, category theory, differential and algebraic geometry, mathematical physics. They all share the common feature of being defined by operations with several inputs and one single output (the associative product, the Lie bracket, the Poisson bracket). A powerful device to handle such kind of algebraic structures is the notion of operad, which have proven to be a fundamental tool to study algebras such as the aforementioned examples, feeding back important outcomes in these various fields. However, algebraic structures
equipped not only with products but also with coproducts play a crucial role in various places in mathematics. One could mention for instance the following important examples: Hopf algebras in representation theory and mathematical physics, Frobenius algebras encompassing the Poincaré duality phenomenon in algebraic topology (which corresponds to unitary and counitary Frobenius bialgebras, see \cite{Koc}), Lie bialgebras introduced by Drinfeld in quantum group theory (see \cite{Dri1} and \cite{Dri2}), involutive Lie bialgebras originally encoding operations on free loops on surfaces in the work of Turaev \cite{Tur} and then generalized to higher dimensional manifolds by Chas-Sullivan \cite{CS2} in string topology \cite{CS1}.
A convenient way to handle such kind of structures is to use the formalism of props, a generalization of operads encoding algebraic structures based on operations with several inputs and several outputs. A dg prop is a collection of complexes $P=\{P(m,n)\}_{m,n\in\mathbb{N}}$, where each $P(m,n)$ represents formal operations with $m$ inputs and $n$ outputs. This collection $P$ is equipped with composition products grafting and concatenating these operations in a compatible way.

This paper is a follow-up of \cite{Yal}, where we set up a homotopy theory for algebras over (possibly colored) differential graded (dg for short) props. The crux of our approach lies on the proof that the Dwyer-Kan classifying spaces attached to categories of algebras over dg props are homotopy invariants of the dg prop. Such spaces have been introduced by Dwyer and Kan in their seminal work on simplicial localization of categories (see \cite{DK1}, \cite{DK2} and \cite{DK3}).
Recall from these papers that the classifying space of a category with weak equivalences (for instance a model category) is the nerve of its subcategory of weak equivalences. It encodes information about homotopy types and internal symmetries of the objects, i.e. their homotopy automorphisms. The goal of the present paper is to give another description of these classifying spaces, in terms of function spaces of dg prop morphisms, in order to make their homotopy theory accessible to computation. These function spaces are moduli spaces of algebra structures, that is, simplicial sets $P\{X\}$ whose vertices are dg prop morphisms $P\rightarrow End_X$ representing a $P$-algebra structure on an object $X$ of the base category.
For us, the base category is the category $\Ch$ of $\mathbb{Z}$-graded chain complexes over a field $\mathbb{K}$.
Let $\Ch^P$ be the category of $P$-algebras and $w\Ch^P$ be its subcategory obtained by restriction to morphisms which are quasi-isomorphisms in $\Ch$. Let us denote by $\mathcal{N}(-)$ the nerve of a category. Our main theorem reads:
\begin{thm}
Let $P$ be a cofibrant dg prop defined in the category of chain complexes $\Ch$ and $X\in \Ch$.
The commutative square
\[
\xymatrix{P\{X\}\ar[d]\ar[r] & \mathcal{N}w\Ch^{P}\ar[d]\\
\{X\}\ar[r] & \mathcal{N}w\Ch
}
\]
is a homotopy pullback of simplicial sets.
\end{thm}
As a consequence, we get the following decomposition of function spaces in terms of homotopy automorphisms:
\begin{cor}
\emph{We have
\[
P\{X\}\sim \coprod_{[X]} Lw\Ch^{P}(X,X)
\]
where $L(-)$ is the simplicial localisation functor of Dwyer-Kan \cite{DK2},
and $[X]$ ranges over the weak equivalence classes of $P$-algebras having $X$ as underlying object.
In particular, the simplicial monoids of homotopy automorphisms $Lw\Ch^{P})(X,X)$ are homotopically small in the sense of Dwyer-Kan, that is, their homotopy groups are all small as sets.}
\end{cor}
Theorem 0.1 is a broad generalization of the first main result of Rezk's thesis \cite[Theorem 1.1.5]{Rez},
which concerns the case of operads in simplicial sets and simplicial modules. However, the method
of \cite{Rez} relies on the existence of a model category structure on algebras over operads, which does not exist anymore
for algebras over dg props.
The crux of the proof of Theorem 0.1 lies in the construction of functorial diagram factorizations in categories of algebras over dg props. We use a new approach, relying on universal categories of algebras over dg props, to perform such constructions
in our context. This method enables us to get round the lack of model structure.

We would like to emphasize some links with two other objects
encoding algebraic structures and their deformations. Theorem 0.1 asserts that we can use
a function space of dg props, the moduli space $P \{ X \}$, to determine
classifying spaces of categories of algebras over dg props $\mathcal{N}w\Ch^{P}$.
The homotopy groups of this moduli space, in turn, can be approached by means of a Bousfield-Kan type spectral sequence.
The $E_2$-page of this spectral sequence is identified with the cohomology of certain deformation complexes.
These complexes have been studied  in \cite{FMY}, \cite{Mar2} and \cite{MV}.
These papers prove the existence of an $L_{\infty}$-structure on such complexes
which generalizes the intrinsic Lie bracket of Schlessinger and Stasheff \cite{SS}.
We aim to apply this spectral sequence technique and provide
new results about the deformation theory of bialgebras in an ongoing work. To complete this outlook, let us point out
that Ciocan-Fontanine and Kapranov used a similar approach to that of Rezk in \cite{CFK} in order to define
a derived moduli space of algebras structures in the formalism of dg schemes. The author recently proved in \cite{Yal3},
by different methods, that the simplicial moduli spaces considered in the present paper are also the global points of derived
moduli stacks in the setting of Toen-Vezzosi's derived algebraic geometry, and that the deformation complexes of \cite{MV} really
computes the tangent complexes of these stacks.

\noindent
\textit{Organization.} In Section 1, we briefly recall several properties of dg props and their algebras, and we define the notion of moduli space of algebraic structures.
In Section 2, we revisit the notion of a colored dg prop as a symmetric monoidal dg category generated by words of colors. Then we perform the main technical construction of
this section: a dg category associated to the data of a small category $\mathcal{J}$ and a colored dg prop $P_{\mathcal{I}}$ encoding $\mathcal{I}$-diagrams of $P$-algebras,
where $\mathcal{I}$ is a subcategory of $\mathcal{J}$. This category of ``formal variables'' is used to explain how a functorial $\mathcal{I}$-diagram of $P$-algebras
can be extended to a functorial $\mathcal{J}$-diagram of $P$-algebras under several technical assumptions. This construction applies in particular
to the functorial factorizations of morphisms provided by the axioms of model categories. In Section 3,
we prove that the classifying space of quasi-isomorphisms of $P$-algebras is weakly equivalent to the classifying space of acyclic surjections of $P$-algebras.
For this, we need to examine in Section 3 how the internal and external tensor products of a diagram category behaves with respect to its injective and projective model structures.
The projective case is more subtle and does not appear in the literature. Then we combine the results of Section 2, those of Section 3 and Quillen's Theorem A to
provide this weak equivalence (induced by an inclusion of categories). Finally, in Section 4 we rely on the previous results to generalize \cite[Theorem 1.1.5]{Rez} to the dg prop setting.

\noindent
\textbf{Acknowledgements.} I would like to thank Benoit Fresse for his useful remarks. I also thanks the referee for his careful reading and useful comments.

\section{Props, algebras and moduli spaces}

Throughout this paper, we work in the category $\Ch$ of  $\mathbb{Z}$-graded chain complexes over a field $\mathbb{K}$.
We write ``dg'' as a shortcut for ``differential graded''.
We give brief recollections on our conventions and on the main definitions concerning dg props in this section. We refer
to \cite{Fre1} for a more comprehensive account.

\subsection{Recollections on props and their algebras}

A $\mathbb{S}$-biobject in $\Ch$ is a double sequence $\{M(m,n)\in \Ch\}_{(m,n)\in\mathbb{N}^{2}}$
where each $M(m,n)$ is equipped with a right action of the symmetric group on $m$ letters $\Sigma_m$,
a left action of the the symmetric group on $n$ letters $\Sigma_n$ and so that these actions commute to each other.
\begin{defn}
A dg prop is a $\mathbb{S}$-biobject in $\Ch$ endowed with associative horizontal composition products
\[
\circ_{h}:P(m_{1},n_{1})\otimes P(m_{2},n_{2})\rightarrow P(m_{1}+m_{2},n_{1}+n_{2}),
\]
vertical associative composition products
\[
\circ_{v}:P(k,n)\otimes P(m,k)\rightarrow P(m,n)
\]
and units $\eta:\mathbb{K}\rightarrow P(n,n)$.
These products satisfy the exchange law
\[
(f_1\circ_h f_2)\circ_v(g_1\circ_h g_2) = (-1)^{|g_1||f_2|}(f_1\circ_v g_1)\circ_h(f_2\circ_v g_2)
\]
and are compatible with the actions of symmetric groups and with the differentials.
Morphisms of dg props are equivariant morphisms of collections compatible with the composition products.
We denote by $Prop$ the category of dg props.
\end{defn}
The following definition shows how a given dg prop encodes algebraic operations on the tensor powers
of a chain complex:
\begin{defn}
(1) The endomorphism dg prop of a chain complex $X$ is given by $End_{X}(m,n)=Hom_{\Ch}(X^{\otimes m},X^{\otimes n})$
where $Hom_{\Ch}(-,-)$ is the internal hom bifunctor of $\Ch$.
The horizontal composition is given by the tensor product of homomorphisms
and the vertical composition is given by the composition of homomorphisms.

(2) Let $P$ be a dg prop. A $P$-algebra is a chain complex $X$ equipped with a dg prop morphism $P\rightarrow End_X$.
\end{defn}
Hence any ``abstract'' operation of $P$ is sent to an operation on $X$, and the way abstract operations
compose under the composition products of $P$ tells us the relations satisfied by the corresponding
algebraic operations on $X$.

One can perform similar constructions in the category of colored $\mathbb{S}$-biobjects
in order to define colored dg props and their algebras:
\begin{defn}
Let $C$ be a non-empty set, called the set of colors.

(1) A $C$-colored $\mathbb{S}$-biobject $M$ is a double sequence of
chain complexes
\[
\{M(m,n)\}_{(m,n)\in\mathbb{N}^{2}}
\]
where each $M(m,n)$ admits commuting right $\Sigma_m$ action and left
$\Sigma_n$ action as well as a decomposition
\[
M(m,n)=\bigoplus_{c_{i},d_{i}\in C}M(c_{1},\cdots,c_{m};d_{1},\cdots,d_{n})
\]
compatible with these actions. The objects $M(c_1,\cdots,c_m;d_1,\cdots,d_n)$
should be thought as spaces of operations with colors $c_1,\cdots,c_m$
indexing the $m$ inputs and colors $d_{1},\cdots,d_{n}$ indexing the
$n$ outputs.

(2) A $C$-colored dg prop $P$ is a $C$-colored $\mathbb{S}$-biobject
endowed with a horizontal composition
\begin{multline*}
\circ_{h}:P(c_{11},\cdots,c_{1m_{1}};d_{11},\cdots,d_{1n_{1}})\otimes\cdots\otimes P(c_{k1},\cdots,c_{km_{k}};d_{k1},\cdots,d_{kn_{1}})\\
\rightarrow P(c_{11},\cdots,c_{km_{k}};d_{k1},\cdots,d_{kn_{k}})\\
\subseteq P(m_{1}+\cdots+m_{k},n_{1}+\cdots+n_{k})\\
\end{multline*}
and a vertical composition
\begin{multline*}
\circ_{v}:P(c_{1},\cdots,c_{k};d_{1},\cdots,d_{n})\otimes P(a_{1},\cdots,a_{m};b_{1},\cdots,b_{k})\\
\rightarrow P(a_{1},\cdots,a_{m};d_{1},\cdots,d_{n})\\
\subseteq P(m,n)\\
\end{multline*}
which is equal to zero unless $b_{i}=c_{i}$ for $1\leq i\leq k$.
These two compositions satisfy associativity axioms (we refer the
reader to \cite{JY} for details).
\end{defn}

\begin{defn}
(1) Let $\{X_{c}\}_{C}$ be a collection of chain complexes.
The $C$-colored endomorphism dg prop $End_{\{X_{c}\}_{C}}$ is defined
by
\[
End_{\{X_{c}\}_{C}}(c_{1},\cdots,c_{m};d_{1},\cdots,d_{n})=Hom_{\Ch}(X_{c_{1}}\otimes\cdots\otimes X_{c_{m}},X_{d_{1}}\otimes\cdots\otimes X_{d_{n}}).
\]

(2) Let $P$ be a $C$-colored dg prop. A $P$-algebra is the data of
a collection of objects $\{X_{c}\}_{C}$ and a $C$-colored dg prop morphism
$P\rightarrow End_{\{X_{c}\}_{C}}$.
\end{defn}

\begin{rem}
Let $\mathcal{I}$ be a small category and let $P$ be a dg prop. We can
build an $ob(\mathcal{I})$-colored dg prop $P_{\mathcal{I}}$ such that the $P_{\mathcal{I}}$-algebras
are the $\mathcal{I}$-diagrams of $P$-algebras in $\Ch$ in the same
way as that of \cite{Mar1}. We refer the reader to Definition 2.3 where this construction is made explicit.
\end{rem}

We provide $\Ch$ with the model category structure such that the weak-equivalences are quasi-isomorphisms and fibrations are degreewise surjections. The category of $\mathbb{S}$-biobjects is a diagram category over $\Ch$, so it inherits a cofibrantly
generated model category structure in which weak equivalences and fibrations are
defined componentwise. The free dg prop functor \cite[Appendix A]{Fre1} gives rise to an adjunction $\Ch^{\mathbb{S}}\rightleftarrows Prop$ between the category of $\mathbb{S}$-biobjects $\Ch^{\mathbb{S}}$ and the category of dg props $Prop$,
which transfers this model category structure to the category of dg props:
\begin{thm}[{see \cite[Theorem 4.9]{Fre1} and \cite[Theorem 1.1]{JY}}]

(1) Suppose that $char(\mathbb{K})>0$. The category $Prop_{0}$
of dg props with non-empty inputs (or outputs) equipped with the
classes of componentwise weak equivalences and componentwise fibrations forms a cofibrantly generated semi-model category.

(2) Suppose that $char(\mathbb{K})=0$. Then the entire category of dg props inherits a cofibrantly generated model category
structure with the weak equivalences and fibrations as above.

(3) Suppose that $char(\mathbb{K})=0$. Let $C$ be a non-empty set.
Then the category $Prop_{C}$ of $C$-colored dg props forms a cofibrantly generated model category with fibrations
and weak equivalences defined componentwise.
\end{thm}
A semi-model category structure is a slightly weakened version of a model category structure where the
lifting axioms only work for cofibrations with cofibrant domain, and where the factorization axioms only work for
a map with a cofibrant domain (see the relevant section of \cite{Fre1}).
The notion of a semi-model category is sufficient to apply the usual constructions of homotopical algebra.
A dg prop $P$ has non-empty inputs if it satisfies
\[
P(0,n)=\begin{cases}
\mathbb{K}, & \text{if $n=0$},\\
0 & \text{otherwise}.
\end{cases}
\]
We define in a symmetric way a dg prop with non-empty outputs.
Such dg props usually encode algebraic structures without unit or without counit, for instance Lie bialgebras.

We will use all the time the existence of a (semi)-model category structure on dg props.
Our results hold over a field of any characteristic: we can work alternatively with every dg prop in characteristic zero or with dg props with non-empty inputs/outputs in positive characteristic.

Finally, we recall from \cite{Fre1} the construction of the endomorphism dg prop associated to a diagram $F:\mathcal{J}\rightarrow \Ch$:
\[
End_F(m,n):=\int_{i\in\mathcal{J}}Hom_{\Ch}(X_{i}^{\otimes m},X_{i}^{\otimes n})
\]
where $X_i=F(i)$.
This end can equivalently be defined as a coreflexive equalizer
\[
\xymatrix{End_F(m,n)\ar[r] & \prod_{i\in\mathcal{J}}Hom_{\Ch}(X_{i}^{\otimes m},X_{i}^{\otimes n})\ar@<1ex>[r]^-{d_{0}}\ar@<-1ex>[r]_-{d_{1}} & \prod_{u:i\rightarrow j}Hom_{\Ch}(X_{i}^{\otimes m},X_{j}^{\otimes n})\ar@/^{2pc}/[l]^{s_{0}}}
\]
where $d_{0}$ is the product of the maps
\[
u_{*}=(F(u)^{\otimes n}\circ -):Hom_{\Ch}(X_{i}^{\otimes m},X_{i}^{\otimes n})\rightarrow Hom_{\Ch}(X_{i}^{\otimes m},X_{j}^{\otimes n})
\]
induced by the morphisms $u:i\rightarrow j$ of $\mathcal{J}$ and $d_{1}$
is the product of the maps
\[
u^{*}=(-\circ F(u)^{\otimes m}):Hom_{\Ch}(X_{j}^{\otimes m},X_{j}^{\otimes n})\rightarrow Hom_{\Ch}(X_{i}^{\otimes m},X_{j}^{\otimes n})
\]
The section $s_{0}$ is the projection on the factors associated to
the identities $id:i\rightarrow i$.
This construction allows us to characterize a diagram of $P$-algebras $F:\mathcal{J}\rightarrow \Ch^P$, where $\Ch^P$ is the category of $P$-algebras in chain complexes, as a dg prop morphism
\[
P\rightarrow End_{U(F)},
\]
where $U(F)$ is the diagram of chain complexes underlying $F$.

\subsection{Moduli spaces of algebra structures}

Throughout the text, we use the Kan-Quillen model category structure on simplicial sets.
A moduli space of algebra structures over a dg prop $P$, on a given
chain complex $X$, is a simplicial set whose points are
the dg prop morphisms $P\rightarrow End_X$ and connected components are homotopy classes of $P$-algebra structures on $X$.
Such a moduli space can be more generally defined on diagrams of chain complexes.
We then deal with endomorphism dg props of diagrams.
To construct properly such a simplicial set and give its first fundamental
properties, we have to recall some results about cosimplicial and simplicial
resolutions in a model category. For the sake of brevity and clarity,
we refer the reader to \cite[Chapter 16]{Hir} for a complete
treatment of the notions of simplicial resolutions, cosimplicial resolutions
and Reedy model categories.

\begin{defn}
Let $\mathcal{M}$ be a model category and let $X$ be an object of $\mathcal{M}$.

(1) A \emph{cosimplicial resolution} of $X$ is a cofibrant approximation to
the constant cosimplicial object $cc_* X$ in the Reedy model category structure on
cosimplicial objects $\mathcal{M}^{\Delta}$ of $\mathcal{M}$.

(2) A \emph{simplicial resolution} of $X$ is a fibrant approximation to
the constant simplicial object $cs_* X$ in the Reedy model category structure on
simplicial objects $\mathcal{M}^{\Delta^{op}}$ of $\mathcal{M}$.
\end{defn}
\begin{defn}
Let $\mathcal{M}$ be a model category and let $X$ be an object of $\mathcal{M}$.

(1) A \emph{cosimplicial frame} on $X$ is a cosimplicial object $\tilde{X}$ in $\mathcal{M}$,
together with a weak equivalence $\tilde{X} \rightarrow cc_* X$ in the Reedy model category
structure of $\mathcal{M}^{\Delta}$. It has to satisfy the two following properties : the induced map $\tilde{X}^0 \rightarrow X$
is an isomorphism, and if $X$ is cofibrant in $\mathcal{M}$ then $\tilde{X}$ is cofibrant in $\mathcal{M}^{\Delta}$.

(2) A \emph{simplicial frame} on $X$ is a simplicial object $\tilde{X}$ in $\mathcal{M}$,
together with a weak equivalence $cs_* X \rightarrow \tilde{X}$ in the Reedy model category
structure of $\mathcal{M}^{\Delta}$. It has to satisfy the following two properties : the induced map $X \rightarrow \tilde{X}^0$
is an isomorphism, and if $X$ is fibrant in $\mathcal{M}$ then $\tilde{X}$ is fibrant in $\mathcal{M}^{\Delta^{op}}$.
\end{defn}
\begin{prop}[{\cite[Proposition 16.1.9]{Hir}}]
Let $\mathcal{M}$ be a model category. There exists functorial simplicial resolutions and functorial
cosimplicial resolutions in $\mathcal{M}$.
\end{prop}
\begin{prop}[{\cite[Proposition 16.6.3]{Hir}}]
Let $X$ be an object of $\mathcal{M}$.

(1) If $X$ is cofibrant then every cosimplicial frame of $X$ is
a cosimplicial resolution of $X$.

(2) If $X$ is fibrant then every simplicial frame of $X$ is a simplicial
resolution of $X$.
\end{prop}

In a model category $\mathcal{M}$, one can define homotopy mapping spaces $Map_{\mathcal{M}}(-,-)$, which are simplicial sets
equipped with a composition law defined up to homotopy. There are two possible definitions.
We can take either $Map_{\mathcal{M}}(X,Y)=Mor_{\mathcal{M}}(X\otimes\Delta^{\bullet},Y)$ where $(-)\otimes\Delta^{\bullet}$ is a cosimplicial resolution,
or $Map_{\mathcal{M}}(X,Y)=Mor_{\mathcal{M}}(X,Y^{\Delta^{\bullet}})$ where $(-)^{\Delta^{\bullet}}$ is a simplicial resolution.
When $X$ is cofibrant and $Y$ is fibrant, these two definitions give the same homotopy type of mapping space
and have also the homotopy type of Dwyer-Kan's hammock localization $L^H(\mathcal{M},w\mathcal{M})(X,Y)$ where $w\mathcal{M}$ is the subcategory of weak equivalences of $\mathcal{M}$ (see \cite{DK3}). Moreover, the set of connected components $\pi_0Map_{\mathcal{M}}(X,Y)$ is the set of homotopy classes $[X,Y]_{\mathcal{M}}$ in $Ho(\mathcal{M})$.

\begin{prop}[{\cite[Corollary 16.5.3]{Hir} and \cite[Corollary 16.5.4]{Hir}}]
Let $\mathcal{M}$ be a model category and $C$ a cosimplicial resolution in $\mathcal{M}$.

(1) If $Y$ is a fibrant object of $\mathcal{M}$, then $Mor_{\mathcal{M}}(C,Y)$
is a fibrant simplicial set.

(2) If $p:X\twoheadrightarrow Y$ is a fibration in $\mathcal{M}$,
then $p_{*}:Mor_{\mathcal{M}}(C,X)\twoheadrightarrow Mor_{\mathcal{M}}(C,Y)$
is a fibration of simplicial sets, acyclic if $p$ is so.

(3) If $p:X\stackrel{\sim}{\rightarrow}Y$ is a weak equivalence of
fibrant objects in $\mathcal{M}$, then $p_{*}:Mor_{\mathcal{M}}(C,X)\stackrel{\sim}{\rightarrow} Mor_{\mathcal{M}}(C,Y)$
is a weak equivalence of fibrant simplicial sets.
\end{prop}

\begin{defn}
Let $P$ be a cofibrant dg prop in $\Ch$.
Let $X$ be a chain complex. The \emph{moduli space of $P$-algebra structures} on $X$
is the simplicial set defined by
\[
P\{X\}=Mor_{Prop}(P\otimes\Delta^{\bullet},End_X),
\]
where $(-)\otimes\Delta^{\bullet}$ is a functorial cosimplicial frame on $Prop$.
We get a functor
\begin{align*}
Prop\rightarrow & sSet\\
P\mapsto & P\{X\}
\end{align*}
where $sSet$ is the category of simplicial sets.
\end{defn}
We can already get two interesting properties of these moduli
spaces:
\begin{lem}
Let $P$ be a cofibrant dg prop. For any chain complex $X$, the moduli space
$P\{X\}$ is a fibrant simplicial set.
\end{lem}
\begin{proof}
Every chain complex is fibrant, and fibration of dg props are defined componentwise, so $End_X$
is a fibrant dg prop. Given that $P$ is cofibrant, the mapping space $P\{X\}$ is
fibrant.
\end{proof}
In this case, the connected components of this moduli space are exactly the homotopy classes of $P$-algebra structures on $X$.

To conclude, let us note that these moduli spaces are a well defined homotopy invariant of algebraic structures over a given object:
\begin{lem}
Let $X$ be a chain complex. Every weak
equivalence of cofibrant dg props $P\stackrel{\sim}{\rightarrow}Q$ gives
rise to a weak equivalence of fibrant simplicial sets
\[
Q\{X\}\stackrel{\sim}{\rightarrow}P\{X\}.
\]
\end{lem}
\begin{proof}
Let $\varphi:P\rightarrow Q$ be a weak equivalence of cofibrant dg props.
According to \cite[Proposition 16.1.24]{Hir}, the map $\varphi$
induces a Reedy weak equivalence of cosimplicial resolutions $P\otimes\Delta^{\bullet}\stackrel{\sim}{\rightarrow}Q\otimes\Delta^{\bullet}$.
The dg prop $End_X$ is fibrant,
so we conclude by \cite[Corollary 16.5.5]{Hir} that this weak equivalence of cosimplicial resolutions
induces a weak equivalence between the corresponding moduli spaces.
\end{proof}

\begin{rem}
The reader may have noticed that, using the existence of functorial cosimplicial resolutions,
Definition 1.12, Lemma 1.13 and Lemma 1.14 could have been stated without the cofibrancy assumption on $P$. In this case, let
\[
P^{\bullet}\stackrel{\sim}{\rightarrow} cc^{\bullet}P
\]
be such a cosimplicial resolution of a dg prop $P$, and
\[
\widetilde{P\{X\}}=Mor_{Prop}(P^{\bullet},End_X)
\]
be this alternative construction of the moduli space.
Let
\[
P_{\infty}\stackrel{\sim}{\rightarrow}P
\]
be a functorial cofibrant resolution of $P$. Then a cosimplicial frame on $P_{\infty}$ is a cosimplicial resolution of $P_{\infty}$ by Proposition 1.10, hence a cosimplicial resolution of $P$ as well. By \cite[Proposition 16.1.13]{Hir}, any two cosimplicial resolutions of a given object are related by a zigzag whose middle object is a fibrant cosimplicial resolution,
and by \cite[Corollary 16.5.5]{Hir} a Reedy weak equivalence of cosimplicial resolutions induce a weak equivalence of mapping spaces, hence
\[
\widetilde{P\{X\}}=Mor_{Prop}(P^{\bullet},End_X) \simeq Mor_{Prop}(P_{\infty}\otimes\Delta^{\bullet},End_X) = P_{\infty}\{X\}.
\]
Our alternative construction of a moduli space directly from a dg prop $P$ thus has the homotopy type of the moduli space of homotopy $P$-algebra structures constructed in Definition 1.12 from a cofibrant resolution of $P$.
\end{rem}

\section{Dg categories associated to colored dg props}

\subsection{Colored dg props as symmetric monoidal dg categories}

We revisit the definition of colored dg props by explaining how they can alternatively be defined as symmetric monoidal dg categories ``monoidally'' generated by the set of colors.
We start with two simple examples before explaining the general construction.
\begin{example}
Any dg prop in $\Ch$ can alternatively be defined as a dg monoidal category $\mathtt{cat}(P)$
such that $ob(\mathtt{cat}(P))=\{x^{\otimes n},n\in\mathbb{N}\}$ (where $x$ is a formal variable), the tensor product
is given by $x^{\otimes m}\otimes x^{\otimes n}=x^{\otimes (m+n)}$ and the complexes of morphisms by
\[
Hom_{\mathtt{cat}(P)}(x^{\otimes m},x^{\otimes n})=P(m,n).
\]
The category of $P$-algebras consists of enriched symmetric monoidal functors
\[
\mathtt{cat}(P)\rightarrow \Ch
\]
with their natural transformations.
\end{example}

\begin{example}
Let $P$ be a ($1$-colored) dg prop. There exists a $2$-colored dg prop $P_{x\rightarrow y}$
such that the category of $P_{x\rightarrow y}$-algebras is the category of morphisms $f:X\rightarrow Y$
in the category of $P$-algebras $\Ch^P$.
It has two colors $x,y$ and it is generated for the composition products by
$P(x,\cdots,x;x,\cdots,x)=P(m,n)$, by $P(y,\cdots,y;y,\cdots,y)=P(m,n)$, and by an element $f\in P(x,y)$ which represents
an abstract arrow $f:x\rightarrow y$.
The associated dg monoidal category $\mathtt{cat}(P_{x\rightarrow y})$ is defined in the following way. Let $Free_{mon}(x,y)$ be the monoid freely generated by the two generators $x$ and $y$, i.e the set of words in two letters $w\in Free_{mon}(x,y)$.
Then the objects of $\mathtt{cat}(P_{x\rightarrow y})$ are the "monoidal words"
\[
ob(\mathtt{cat}(P_{x\rightarrow y}))=\{w_{\otimes}(x,y),w\in Free_{mon}(x,y)\}
\]
where $w_{\otimes}(x,y)$ is the formal tensor product corresponding to the word $w$.
The complexes of morphisms are
\[
Hom_{\mathtt{cat}(P_{x\rightarrow y})}(w_{\otimes}(x,y), v_{\otimes}(x,y))=P_{x\rightarrow y}(\underline{w},\underline{v})
\]
where $\underline{w}$ is the ordered sequence of letters, i.e colors, appearing in the word $w$.
Algebras over $P_{x\rightarrow y}$ are enriched symmetric monoidal functors $\mathtt{cat}(P_{x\rightarrow y})\rightarrow \Ch$.
A $P_{x\rightarrow y}$-algebra is equivalent to a diagram of $P$-algebras $\{\bullet\rightarrow\bullet\}\rightarrow \Ch^P$.
\end{example}
These constructions can be generalized to arbitrary diagrams
\begin{defn}
Let $\mathcal{I}$ be a small category. Then there exists a $ob(\mathcal{I})$-colored dg prop $P_{\mathcal{I}}$
consisting of abstract objects $x_i$ associated to  $i\in\mathcal{I}$, and the morphisms of $P_{\mathcal{I}}$ are generated by operations $p\in P_{\mathcal{I}}(x_i^{\otimes m},x_i^{\otimes n})$ associated to each $p\in P(m,n)$ and each variable $x_i$, as well as abstract arrows $f: x_i\rightarrow x_j$ associated to the morphisms of $\mathcal{I}$.
The corresponding dg monoidal category $\mathtt{cat}(P_{\mathcal{I}})$ is defined as follows:
\[
ob(\mathtt{cat}(P_{\mathcal{I}}))=\{w_{\otimes}(x_i,i\in ob(\mathcal{I})),w\in Free_{mon}(x_i,i\in ob(\mathcal{I}))\}.
\]
The tensor product is defined by
\[
w_{\otimes}(x_i,i\in ob(\mathcal{I}))\otimes v_{\otimes}(x_i,i\in ob(\mathcal{I}))=(w*v)_{\otimes}(x_i,i\in ob(\mathcal{I})).
\]
The complexes of morphisms are
\[
Hom_{\mathtt{cat}(P_{\mathcal{I}})}(w_{\otimes}(x_i,i\in ob(\mathcal{I})), v_{\otimes}(x_i,i\in ob(\mathcal{I})))=P_{\mathcal{I}}(\underline{w},\underline{v}).
\]
The composition on the dg hom is the vertical composition product of $P_{\mathcal{I}}$, and the tensor
product of morphisms is the horizontal composition product of $P_{\mathcal{I}}$.
\end{defn}
In other words, the category $\mathtt{cat}(P_{\mathcal{I}})$ is a differential graded monoidal category monoidally generated
on objects by the set of colours of $P_{\mathcal{I}}$. This can be generalized in any symmetric monoidal category, giving
an alternative definition of a colored dg prop:
\begin{defn}
(1) A $C$-colored dg prop is a small symmetric monoidal dg category monoidally generated by $C$.

(2) A $P_{\mathcal{I}}$-algebra is a symmetric monoidal dg functor $\mathtt{cat}(P_{\mathcal{I}})\rightarrow \Ch$.
\end{defn}
\begin{prop}
A $P_{\mathcal{I}}$-algebra corresponds to an $\mathcal{I}$-diagram of $P$-algebras.
\end{prop}
This result follows from the construction of $P_{\mathcal{I}}$ in terms of generators and relations.
For more details we refer the reader to \cite[Section 2]{Mar1}, where such a construction is carried out in the case of colored dg operads.

\subsection{Categories of universal twisted sums and functorial diagrams of algebras}

Let $P_{\mathcal{J}}$ be a colored prop on a small category $\mathcal{J}$.
The category $\mathtt{cat}(P_{\mathcal{J}})$ reflects the universal structures of the symmetric monoidal category defined by a $P$-algebra in the category of chain complexes.
But for some constructions of homotopy theory, we need operations of the ambient category of chain complexes which lie outside the image of this category $\mathtt{cat}(P_{\mathcal{J}})$. Namely, we need to perform direct sums $C\oplus D$, suspensions $\Sigma C$, and twisted complexes $(C,d)$ which we form by adding a twisting homomorphism $d\in Hom(C,C)$ to the internal differential of a chain complex $\delta: C\rightarrow C$. These operations can clearly not be formed within the image of $\mathtt{cat}(P_{\mathcal{J}})$ in the category of chain complexes in general.
Therefore, we define a universal enriched category $\mathtt{TwSum}(P_{\mathcal{J}})$ generated by the formal image of the tensor products $w(x_j,j\in \mathcal{J})\in ob(\mathtt{cat}(P_{\mathcal{J}}))$ under such direct sum, suspension, and twisting operations. If we put all these operations together, then we get the notion of a twisted direct sum which we formalize in our definition. Let us simply mention that we use formal tensor products $\mathbb{K} e\otimes V$, where $\mathbb{K} e$ is the free $\mathbb{K}$-module spanned by a homogeneous element of degree $d = \deg(e)$ to materialize a $d$-fold suspension operation $\Sigma^d: C\mapsto\Sigma^d C$.
In the sequel, our idea is to define universal models of the homotopical construction which we need to work out our problems in this enriched category $\mathtt{TwSum}(P_{\mathcal{J}})$.

\subsubsection{Construction of the category of universal twisted sums}

Let $\mathcal{J}$ be a small category and $P_{\mathcal{J}}$ the associated $ob(\mathcal{J})$-colored dg prop.
Our goal is to build from $\mathtt{cat}(P_{\mathcal{J}})$ a certain dg monoidal category $\mathtt{TwSum}(P_{\mathcal{J}})$ called
its category of universal twisted sums. The objects are pairs
\[
\left( \bigoplus_{\underline{\alpha}\in A} (\mathbb{K}\mathbb{e}_{\underline{\alpha}})\otimes (x_{\alpha_1}\otimes\cdots\otimes x_{\alpha_n}),\mathtt{tw} \right)
\]
where
\begin{itemize}
\item the first term $\bigoplus_{\underline{\alpha}\in A}$ is a formal sum over a finite set $A$ of multi-indices $\alpha = 
(\alpha_1, ...,\alpha_n)$ of formal tensor products
$(\mathbb{K}\mathbb{e}_{\underline{\alpha}})\otimes (x_{\alpha_1}\otimes\cdots\otimes x_{\alpha_n})$, where
$x_{\alpha_1}\otimes\cdots\otimes x_{\alpha_n}$ is an object of $\mathtt{cat}(P_{\mathcal{J}})$ and we consider the graded $\mathbb{K}$-module $\mathbb{K}\mathbb{e}_{\underline{\alpha}}$
generated by a homogeneous element $\mathbb{e}_{\underline{\alpha}}$ of a certain degree $d_{\underline{\alpha}}=deg(\mathbb{e}_{\underline{\alpha}})$;

\item The second term represents a collection of homomorphisms
\begin{equation*}
\mathtt{tw}_{\underline{\alpha}\,\underline{\beta}}
\in\mathbb{e}_{\underline{\alpha}}\otimes\mathbb{e}_{\underline{\beta}}^{\vee}\otimes Hom_{\mathtt{cat}(P_{\mathcal{J}})}(x_{\beta_1}\otimes\dots\otimes x_{\beta_m},
x_{\alpha_1}\otimes\dots\otimes x_{\alpha_n}),
\end{equation*}
indexed by the couples $(\underline{\alpha},\underline{\beta})\in 
A^2$, homogeneous of degree $-1$, and that satisfy the relation of 
twisting cochains:
\begin{equation*}
\delta(\mathtt{tw}_{\underline{\alpha}\underline{\beta}})+
\sum_{\underline{\gamma}\in 
A}\mathtt{tw}_{\underline{\alpha}\underline{\gamma}}\circ\mathtt{tw}_{\underline{\gamma}\underline{\beta}}=0
\end{equation*}
in the dg-module $\mathbb{e}_{\underline{\alpha}}\otimes\mathbb{e}_{\underline{\beta}}^{\vee}\otimes Hom_{\mathtt{cat}(P_{\mathcal{J}})}(x_{\beta_1}\otimes\dots\otimes x_{\beta_m},
x_{\alpha_1}\otimes\dots\otimes x_{\alpha_n})$,
for every couple of sequences of colours $(\underline{\alpha},\underline{\beta})\in A^2$.
The notation $\mathbb{e}_{\underline{\beta}}^{\vee}$ represents an element dual to $\mathbb{e}_{\underline{\beta}}$,
homogeneous of degree $\deg(\mathbb{e}_{\underline{\beta}}^{\vee}) = -\deg(\mathbb{e}_{\underline{\beta}})$,
and we use the relation $\mathbb{e}_{\underline{\beta}}^{\vee}(\mathbb{e}_{\underline{\beta}}) = 1$
when we form the composites $\mathtt{tw}_{\underline{\alpha}\,\underline{\gamma}}\circ\mathtt{tw}_{\underline{\gamma}\,\underline{\beta}}$.
\end{itemize}
We define the dg-modules of homomorphisms of $\mathtt{TwSum}(P_{\mathcal{J}})$
as the twisted sums of dg-modules
\begin{align*}
Hom_{\mathtt{TwSum}(P_{\mathcal{J}})}(\underbrace{(\oplus_{\underline{\beta}\in B} (\mathbb{K}\mathbb{e}_{\underline{\beta}})\otimes(x_{\beta_1}\otimes\dots\otimes x_{\beta_m}),\mathtt{tw}_L)}_{= L},
\underbrace{(\oplus_{\underline{\alpha}\in A}(\mathbb{K}\mathbb{e}_{\underline{\alpha}})\otimes(x_{\alpha_1}\otimes\dots\otimes x_{\alpha_n}),\mathtt{tw}_K)}_{= K})
\\
:= (\bigoplus_{(\underline{\beta}\,\underline{\alpha})\in B\times A}
\mathbb{K}\mathbb{e}_{\underline{\alpha}}\otimes\mathbb{K}\mathbb{e}_{\underline{\beta}}^{\vee}\otimes Hom_{\mathtt{cat}(P_{\mathcal{J}})}(x_{\beta_1}\otimes\dots\otimes x_{\beta_m},
x_{\alpha_1}\otimes\dots\otimes x_{\alpha_n}),\partial),
\end{align*}
with twisting homomorphism $\partial: (f_{\underline{\beta}\,\underline{\alpha}})\mapsto(\partial(f)_{\underline{\beta}\,\underline{\alpha}})$
such that
\begin{equation*}
\partial(f)_{\underline{\beta}\,\underline{\alpha}}
= \sum_{\underline{\gamma}\in B}\mathtt{tw}_{\underline{\beta}\,\underline{\gamma}}\circ f_{\underline{\gamma}\,\underline{\alpha}}
- \sum_{\underline{\gamma}\in A}sign(f) f_{\underline{\beta}\,\underline{\gamma}}\circ\mathtt{tw}_{\underline{\gamma}\,\underline{\alpha}},
\end{equation*}
for every couple of sequences of colours $(\underline{\alpha},\underline{\beta})$, where $sign(f)$ is a sign depending on $f$.
This endows $\mathtt{TwSum}(P_{\mathcal{J}})$ with a dg category structure:

\begin{proof}
We equip this dg hom $Hom_{\mathtt{TwSum}(P_{\mathcal{J}})}(K,L)$ with the total differential
$\delta +\partial$ where $\delta$ is the internal differential induced by the differential of $P$ and
$\partial$ is the twisting homomorphism. The fact that $(\delta +\partial)^2= 0$
follows from the relation of twisting cochains satisfied by the $\mathtt{tw}$'s with respect to $\delta$.
Indeed, for each $\underline{\beta}\in B,\underline{\alpha}\in A$, we have
\[
(\delta+\partial)^2(f)_{\underline{\beta},\underline{\alpha}} = (\delta(\partial)+\partial^2)(f)_{\underline{\beta},\underline{\alpha}}
\]
where $\delta(\partial)$ is the usual differential of a homomorphism defined by the commutator
\[
\delta(\partial)=\delta\circ\partial - (-1)^{deg(\partial)}\partial\circ\delta = \delta\circ\partial + \partial\circ\delta.
\]
We have
\[
\delta(\partial)(f)_{\underline{\beta},\underline{\alpha}} = \delta(\partial (f)_{\underline{\beta},\underline{\alpha}})+\partial (\delta (f))_{\underline{\beta},\underline{\alpha}} = \delta(\mathtt{tw}_{\underline{\beta},\underline{\alpha}})(f)
\]
and
\begin{eqnarray*}
\partial^2(f)_{\underline{\beta},\underline{\alpha}} & = & \partial (\partial (f))_{\underline{\beta},\underline{\alpha}} \\
 & = & \sum_{\underline{\gamma}\in B}\mathtt{tw}_{\underline{\beta}\,\underline{\gamma}}\circ \partial (f)_{\underline{\gamma}\,\underline{\alpha}}
- \sum_{\underline{\gamma}\in A}sign(\partial(f)) \partial (f)_{\underline{\beta}\,\underline{\gamma}}\circ\mathtt{tw}_{\underline{\gamma}\,\underline{\alpha}} \\
 & = & \sum_{\underline{\gamma}\in B}\mathtt{tw}_{\underline{\beta}\,\underline{\gamma}}\circ \partial(f)_{\underline{\gamma}\,\underline{\alpha}}
+ \sum_{\underline{\gamma}\in A}sign(f) \partial(f)_{\underline{\beta}\,\underline{\gamma}}\circ\mathtt{tw}_{\underline{\gamma}\,\underline{\alpha}} \\
 & = &  \left(\sum_{\underline{\gamma}\in B}\mathtt{tw}_{\underline{\beta}\,\underline{\gamma}}\circ\mathtt{tw}_{\underline{\gamma}\,\underline{\alpha}}\right)(f)
\end{eqnarray*}
because $sign(\partial(f))=sign(f)-1$ (the homomorphism $\partial$ is of degree $-1$),
so
\[
(\delta+\partial)^2(f)_{\underline{\beta},\underline{\alpha}} = \left(\delta(\mathtt{tw}_{\underline{\beta},\underline{\alpha}})+
\sum_{\underline{\gamma}\in B}\mathtt{tw}_{\underline{\beta}\,\underline{\gamma}}\circ\mathtt{tw}_{\underline{\gamma}\,\underline{\alpha}}\right)(f) = 0.
\]
For each object
\[
K=(\oplus_{\underline{\alpha}}\mathbb{K}\mathbb{e}_{\underline{\alpha}}\otimes(x_{\alpha_1}\otimes\dots\otimes x_{\alpha_n}),\mathtt{tw}_K)
\]
of $\mathtt{TwSum}(P_{\mathcal{J}})$, the associated identity element in the dg hom $Hom_{\mathtt{TwSum}(P_{\mathcal{J}})}(K,K)$ is the $0$-cycle defined by
\[
\oplus_{\underline{\alpha}}(\mathbb{K}\mathbb{e}_{\underline{\alpha}})\otimes\mathbb{K}\mathbb{e}_{\underline{\alpha}}^{\vee}\otimes id_{x_{\alpha_1}\otimes\dots\otimes x_{\alpha_n}},
\]
where $id_{x_{\alpha_1}\otimes\dots\otimes x_{\alpha_n}}$ is the identity on the object $x_{\alpha_1}\otimes\dots\otimes x_{\alpha_n}$ of $\mathtt{cat}(P_{\mathcal{J}})$.
The composition law
\[
Hom_{\mathtt{TwSum}(P_{\mathcal{J}})}(K,L)\otimes Hom_{\mathtt{TwSum}(P_{\mathcal{J}})}(L,M)\rightarrow Hom_{\mathtt{TwSum}(P_{\mathcal{J}})}(K,M)
\]
on such dg homs is then defined by the composition of dg homs in $\mathtt{cat}(P_{\mathcal{J}})$ and the relation
$\mathbb{e}_{\underline{\alpha}}^{\vee}(\mathbb{e}_{\underline{\alpha}}) = 1$ on matching colors.
The compatibility of this composition with the twisted differentials of the dg homs is automatic.
\end{proof}

\subsubsection{The tensor structure on a category of universal twisted sums}

The category $\mathtt{TwSum}(P_{\mathcal{J}})$ is equipped with a dg enriched symmetric monoidal structure, defined by the natural distribution formula at the level of objects
\begin{align*}
\underbrace{(\oplus_{\underline{\alpha}} (\mathbb{K}\mathbb{e}_{\underline{\alpha}})\otimes(x_{\alpha_1}\otimes\dots\otimes x_{\alpha_m}),\mathtt{tw}_K)}_{K}
\otimes\underbrace{(\oplus_{\underline{\beta}} ((\mathbb{K}\mathbb{e}_{\underline{\beta}})\otimes(x_{\beta_1}\otimes\dots\otimes x_{\beta_n}),\mathtt{tw}_L)}_{L}
\\
:= \underbrace{(\oplus_{\underline{\alpha},\underline{\beta}} ((\mathbb{K}\mathbb{e}_{\underline{\alpha}}\otimes\mathbb{e}_{\underline{\beta}})
\otimes(x_{\alpha_1}\otimes\dots\otimes x_{\alpha_m}\otimes x_{\beta_1}\otimes\dots\otimes x_{\beta_n}),\mathtt{tw}_K\otimes id + id\otimes\mathtt{tw}_L)}_{= K\otimes L},
\end{align*}
and where we use the horizontal compositions
\begin{multline*}
\begin{aligned} & \mathbb{K}\mathbb{e}_{\underline{\gamma}}\otimes \mathbb{K}\mathbb{e}_{\underline{\alpha}}^{\vee}\otimes Hom_{\mathtt{cat}(P_{\mathcal{J}})}(x_{\alpha_1}\otimes\dots\otimes x_{\alpha_m},
x_{\gamma_1}\otimes\dots\otimes x_{\gamma_p})
\\
\otimes & \mathbb{K}\mathbb{e}_{\underline{\delta}}\otimes\mathbb{K}\mathbb{e}_{\underline{\beta}}^{\vee}\otimes Hom_{\mathtt{cat}(P_{\mathcal{J}})}(x_{\beta_1}\otimes\dots\otimes x_{\beta_n},
x_{\delta_1}\otimes\dots\otimes x_{\delta_q})
\end{aligned}
\\
\xrightarrow{\otimes}\begin{aligned}[t] & (\mathbb{K}\mathbb{e}_{\underline{\gamma}}\otimes\mathbb{K}\mathbb{e}_{\underline{\delta}})
\otimes(\mathbb{K}\mathbb{e}_{\underline{\alpha}}\otimes\mathbb{K}\mathbb{e}_{\underline{\beta}})^{\vee}
\\
\otimes & Hom_{\mathtt{cat}(P_{\mathcal{J}})}(x_{\alpha_1}\otimes\dots\otimes x_{\alpha_m}\otimes x_{\beta_1}\otimes\dots\otimes x_{\beta_n},
x_{\gamma_1}\otimes\dots\otimes x_{\gamma_p}\otimes x_{\delta_1}\otimes\dots\otimes x_{\delta_q})
\end{aligned}
\end{multline*}
to define the formal twisted cochain $\mathtt{tw}_K\otimes id + id\otimes\mathtt{tw}_L$ of this object $K\otimes L$.
An analogous construction holds at the level of homomorphisms.

Each object $x_{\alpha_1}\otimes\dots\otimes x_{\alpha_n}\in \mathtt{cat}(P_{\mathcal{J}})$
is naturally identified with the trivial twisted sum $K = (\mathbb{K}\mathbb{e}_0\otimes(x_{\alpha_1}\otimes\dots\otimes x_{\alpha_n}),0)$ in $\mathtt{TwSum}(P_{\mathcal{J}})$, where $\deg(\mathbb{e}_0) = 0\Rightarrow\mathbb{K} \mathbb{e}_0 = \mathbb{K}$.
In particular, to each $x_{\alpha_i}$ corresponds a trivial twisted sum $K_{\alpha_i} = (\mathbb{K}\mathbb{e}_0\otimes x_{\alpha_i},0)$. This defines a functor
\[
\mathtt{cat}(P_{\mathcal{J}})\rightarrow \mathtt{TwSum}(P_{\mathcal{J}}).
\]
The category of universal twisted sums satisfies the following universal property with respect to this functor:
\begin{lem}
For every symmetric monoidal dg functor $R:\mathtt{cat}(P_{\mathcal{J}})\rightarrow \Ch$ (that is, every
$P_{\mathcal{J}}$-algebra), there exists a canonical factorization
\[
\xymatrix{
\mathtt{cat}(P_{\mathcal{J}}) \ar[d] \ar[r]^-R & \Ch \\
\mathtt{TwSum}(P_{\mathcal{J}}) \ar[ur]_-{\tilde{R}} &
}.
\]
\end{lem}
\begin{proof}
We construct $\tilde{R}$ by setting first $\tilde{R}(K_{\alpha_i})=R(x_{\alpha_i})$ so that the diagram commutes.
Then, for any object
\[
\left( \bigoplus_{\underline{\alpha}} (\mathbb{K}\mathbb{e}_{\underline{\alpha}})\otimes (x_{\alpha_1}\otimes\cdots\otimes x_{\alpha_n}),\mathtt{tw} \right)
\]
of $\mathtt{TwSum}(P_{\mathcal{J}})$, we define
\begin{eqnarray*}
 & \tilde{R}\left( \bigoplus_{\underline{\alpha}} (\mathbb{K}\mathbb{e}_{\underline{\alpha}})\otimes (x_{\alpha_1}\otimes\cdots\otimes x_{\alpha_n}),\mathtt{tw} \right) & \\
 =  & \left( \bigoplus_{\underline{\alpha}} (\mathbb{K}\mathbb{e}_{\underline{\alpha}})\otimes (R(x_{\alpha_1})\otimes\cdots\otimes R(x_{\alpha_n})),R(\mathtt{tw}) \right) &
\end{eqnarray*}
where the left-hand term is built with the direct sum and tensor product of $\Ch$.
The differential of $\tilde{R}\left( \bigoplus_{\underline{\alpha}} (\mathbb{K}\mathbb{e}_{\underline{\alpha}})\otimes (x_{\alpha_1}\otimes\cdots\otimes x_{\alpha_n}),\mathtt{tw} \right)$ is then defined on each component of this direct sum
by the sum of the differential of $R(x_{\alpha_1})\otimes\cdots\otimes R(x_{\alpha_n})$ with a twisting cochain $R(\mathtt{tw})$
defined  as follows.
Since $R$ is a symmetric monoidal dg functor, it induces a morphism of chain complexes
\begin{eqnarray*}
R_{x_{\beta_1}\otimes\dots\otimes x_{\beta_m},x_{\alpha_1}\otimes\dots\otimes x_{\alpha_n}}: & & \\
Hom_{\mathtt{cat}(P_{\mathcal{J}})}(x_{\beta_1}\otimes\dots\otimes x_{\beta_m},x_{\alpha_1}\otimes\dots\otimes x_{\alpha_n}) & & \\
 &  \rightarrow & \\
Hom_{\Ch}(R(x_{\beta_1})\otimes\dots\otimes R(x_{\beta_m}),
R(x_{\alpha_1})\otimes\dots\otimes R(x_{\alpha_n})) & &
\end{eqnarray*}
so that we can well define the collection $R(\mathtt{tw})=\{R(\mathtt{tw}_{\underline{\alpha}\,\underline{\beta}})\}_{\underline{\alpha}\,\underline{\beta}}$ by
\begin{eqnarray*}
R(\mathtt{tw}_{\underline{\alpha}\,\underline{\beta}}) & = & R_{x_{\beta_1}\otimes\dots\otimes x_{\beta_m},x_{\alpha_1}\otimes\dots\otimes x_{\alpha_n}}(\mathtt{tw}_{\underline{\alpha}\,\underline{\beta}}) \\
 & \in & \mathbb{e}_{\underline{\alpha}}\otimes\mathbb{e}_{\underline{\beta}}^{\vee}\otimes Hom_{\Ch}(R(x_{\beta_1})\otimes\dots\otimes R(x_{\beta_m}),
R(x_{\alpha_1})\otimes\dots\otimes R(x_{\alpha_n})).\\
\end{eqnarray*}
This collection satisfies the relation of twisting cochains because $R$ is a symmetric monoidal dg functor and the collection
$\{\mathtt{tw}_{\underline{\alpha}\,\underline{\beta}}\}_{\underline{\alpha}\,\underline{\beta}}$ satisfies the relation of twisting cochains in $\mathtt{TwSum}(P_{\mathcal{J}})$.
\end{proof}

\subsubsection{Functorial diagrams of algebras}

Our purpose is to use categories of universal twisted sums to perform
diagrams of dg $P$-algebras ``functorial in their variables'' in a suitable sense.

Recall that the colored dg prop $P_{\mathcal{J}}$ parametrising $\mathcal{J}$-diagrams of $P$-algebras is equivalent
to the datum of a symmetric monoidal dg category $\mathtt{cat}(P_{\mathcal{J}})$. Algebras over $P_{\mathcal{J}}$ are then
strict symmetric monoidal dg functors $\mathtt{cat}(P_{\mathcal{J}})\rightarrow \Ch$, and morphisms of
$P_{\mathcal{J}}$-algebras are natural transformations preserving the strict symmetric monoidal dg structures.
Such a natural transformation corresponds to a natural transformation of $\mathcal{J}$-diagrams of $P$-algebras.

Now let $A,B:\mathtt{cat}(P_{\mathcal{J}})\rightarrow \Ch$ be two such algebras, and $\phi:A\Rightarrow B$
be a strict symmetric monoidal dg natural transformation. Recall that, according to Lemma 2.6, such functors
lift to strict symmetric monoidal dg functors $\tilde{A},\tilde{B}:\mathtt{TwSum}(P_{\mathcal{J}})\rightarrow \Ch$.
We want to prove that such a lift works similarly for symmetric monoidal dg natural transformations between such functors:
\begin{lem}
The natural transformation $\phi$ lifts to
a strict symmetric monoidal dg natural transformation $\tilde{\phi}:\tilde{A}\Rightarrow\tilde{B}$.
\end{lem}
\begin{proof}
To see this, let us first recall from \cite{Kel} the notion of enriched natural transformation in the case where the categories are enriched over $\Ch$. Let $F,G:\mathcal{C}\rightarrow \mathcal{D}$ be two dg functors and $Hom_{\mathcal{C}}(-,-)$, $Hom_{\mathcal{D}}(-,-)$ be respectively the dg homs of $\mathcal{C}$ and $\mathcal{D}$. A dg natural transformation $\tau:F\Rightarrow G$ is a collection of chain morphisms $\{\tau(x):\mathbb{K}\rightarrow Hom_{\mathcal{D}}(F(x),G(x))\}_{x\in ob(\mathcal{C})}$, that is of $0$-cycles in the $Hom_{\mathcal{D}}(F(x),G(x))$'s indexed by the objects $x$
of $\mathcal{C}$. For every $x,y\in ob(\mathcal{C})$, this collection makes the following diagram commutative:
\[
\xymatrix{
Hom_{\mathcal{C}}(x,y) \ar[d]^-{\cong}\ar[r]^-{\cong} & Hom_{\mathcal{C}}(x,y)\otimes\mathbb{K} \ar[d]^-{G_{x,y}\otimes\tau(x)} \\
\mathbb{K}\otimes Hom_{\mathcal{D}}(x,y) \ar[d]_-{\tau(y)\otimes F_{x,y}} & Hom_{\mathcal{D}}(G(x),G(y))\otimes Hom_{\mathcal{D}}(F(x),G(x)) \ar[d]^-{\circ_{\mathcal{D}}} \\
Hom_{\mathcal{D}}(F(y),G(y))\otimes Hom_{\mathcal{D}}(F(x),F(y)) \ar[r]^-{\circ_{\mathcal{D}}} & Hom_{\mathcal{D}}(F(x),G(y))}.
\]
For any object
\[
K=\left( \bigoplus_{\underline{\alpha}} (\mathbb{K}\mathbb{e}_{\underline{\alpha}})\otimes (x_{\alpha_1}\otimes\cdots\otimes x_{\alpha_n}),\mathtt{tw}_K \right)
\]
of $\mathtt{TwSum}(P_{\mathcal{J}})$, we define the associated $0$-cycle $\tilde{\phi}$ in
\[
Hom_{\Ch}((\bigoplus_{\underline{\alpha}} (\mathbb{K}\mathbb{e}_{\underline{\alpha}})\otimes (A(x_{\alpha_1})\otimes\cdots\otimes A(x_{\alpha_n})),A(\mathtt{tw}_K)),
(\bigoplus_{\underline{\alpha}} (\mathbb{K}\mathbb{e}_{\underline{\alpha}})\otimes (B(x_{\alpha_1})\otimes\cdots\otimes B(x_{\alpha_n})),B(\mathtt{tw}_K)))
\]
by
\[
\tilde{\phi}(K) = \bigoplus_{\underline{\alpha}} (\mathbb{K}\mathbb{e}_{\underline{\alpha}})\otimes (\phi(x_{\alpha_1})\otimes\cdots\otimes \phi(x_{\alpha_n})).
\]
We have to prove that this form a $0$-cycle, thus that
\[
(\delta+B(\mathtt{tw}_K))\circ\tilde{\phi}(K) = \tilde{\phi}(K)\circ(\delta+A(\mathtt{tw}_K)).
\]
The equality
\[
\delta\circ\tilde{\phi}(K) = \tilde{\phi}(K)\circ\delta
\]
follows from the fact that each $\phi(x_{\alpha_i}):A(x_{\alpha_i})\rightarrow B(x_{\alpha_i})$ is a morphism of chain complexes and the differential $\delta$ is the differential of $B(x_{\alpha_1})\otimes\cdots\otimes B(x_{\alpha_n})$ on the left-hand side of the equality and of $A(x_{\alpha_1})\otimes\cdots\otimes A(x_{\alpha_n})$ on the right-hand side.
The equality
\[
B(\mathtt{tw}_K)\circ\tilde{\phi}(K) = \tilde{\phi}(K)\circ A(\mathtt{tw}_K)
\]
follows from the definition of $A(\mathtt{tw}_K)$ as
\[
A(\mathtt{tw}_K)=\{A_{x_{\beta_1}\otimes\dots\otimes x_{\beta_m},x_{\alpha_1}\otimes\dots\otimes x_{\alpha_n}}((\mathtt{tw}_K)_{\underline{\alpha}\,\underline{\beta}})\}_{\underline{\alpha}\,\underline{\beta}}
\]
(and the same for $B(\mathtt{tw}_K)$), the fact that $\phi$ is a dg natural transformation between $A$ and $B$ and the definition of $\tilde{\phi}(K)$ in function of the $\phi(x_{\alpha_i})$'s.

Concerning the monoidality of our collection of $0$-cycles $\{\tilde{\phi}(K)\}_{K\in ob(\mathtt{TwSum}(P_{\mathcal{J}}))}$, recall from 2.2.2 that the tensor product of two objects of $\mathtt{TwSum}(P_{\mathcal{J}})$ is defined by
\begin{align*}
\underbrace{(\oplus_{\underline{\alpha}} (\mathbb{K}\mathbb{e}_{\underline{\alpha}})\otimes(x_{\alpha_1}\otimes\dots\otimes x_{\alpha_m}),\mathtt{tw}_K)}_{K}
\otimes\underbrace{(\oplus_{\underline{\beta}} ((\mathbb{K}\mathbb{e}_{\underline{\beta}})\otimes(x_{\beta_1}\otimes\dots\otimes x_{\beta_n}),\mathtt{tw}_L)}_{L}
\\
:= \underbrace{(\oplus_{\underline{\alpha},\underline{\beta}} ((\mathbb{K}\mathbb{e}_{\underline{\alpha}}\otimes\mathbb{e}_{\underline{\beta}})
\otimes(x_{\alpha_1}\otimes\dots\otimes x_{\alpha_m}\otimes x_{\beta_1}\otimes\dots\otimes x_{\beta_n}),\mathtt{tw}_K\otimes id + id\otimes\mathtt{tw}_L)}_{= K\otimes L}
\end{align*}
and that the functors $\tilde{A},\tilde{B}:\mathtt{TwSum}(P_{\mathcal{J}})\rightarrow \Ch$ associated to $A,B:\mathtt{cat}(P_{\mathcal{J}})\rightarrow \Ch$ are defined by
\begin{eqnarray*}
 & \tilde{A}\left( \bigoplus_{\underline{\alpha}} (\mathbb{K}\mathbb{e}_{\underline{\alpha}})\otimes (x_{\alpha_1}\otimes\cdots\otimes x_{\alpha_n}),\mathtt{tw} \right) & \\
 =  & \left( \bigoplus_{\underline{\alpha}} (\mathbb{K}\mathbb{e}_{\underline{\alpha}})\otimes (A(x_{\alpha_1})\otimes\cdots\otimes A(x_{\alpha_n})),A(\mathtt{tw}) \right) & .
\end{eqnarray*}
We have natural isomorphisms
\[
a_{K\otimes L}:\tilde{A}(K\otimes L)\stackrel{\cong}{\rightarrow} \tilde{A}(K)\otimes \tilde{A}(L)
\]
and
\[
b_{K\otimes L}:\tilde{B}(K\otimes L)\stackrel{\cong}{\rightarrow} \tilde{B}(K)\otimes \tilde{B}(L)
\]
induced by natural isomorphisms
\[
A(\cdot\otimes \cdot)\stackrel{\cong}{\rightarrow} A(\cdot)\otimes A(\cdot)
\]
and
\[
B(\cdot\otimes\cdot)\stackrel{\cong}{\rightarrow} B(\cdot)\otimes B(\cdot)
\]
since $A$ and $B$ are symmetric monoidal functors. We have to check the commutativity of the square below:
\[
\xymatrix{
\tilde{A}(K\otimes L)\ar[d]_-{a_{K\otimes L}}\ar[r]^-{\tilde{\phi}(K\otimes L)} & \tilde{B}(K\otimes L)\ar[d]_-{b_{K\otimes L}} \\
\tilde{A}(K)\otimes \tilde{A}(L)\ar[r]^-{\tilde{\phi}(K)\otimes\tilde{\phi}(L)} & \tilde{B}(K)\otimes \tilde{B}(L)
}.
\]
By construction of $\tilde{A}$, $\tilde{B}$ and $\tilde{\phi}$, which are defined by applying $A$, $B$ and $\phi$ to each variable of the tensor powers defining the objects of $\mathtt{TwSum}(P_{\mathcal{J}})$, this boils down to the commutativity of such a monoidality square for $A$, $B$ and $\phi$, which holds because $\phi$ is a monoidal natural transformation.

The naturality of $\{\tilde{\phi}(K)\}_{K\in ob(\mathtt{TwSum}(P_{\mathcal{J}}))}$ then follows directly from the naturality of $\phi$.
\end{proof}
We consequently get two functors
\[
\tilde{A}_*,\tilde{B}_*:\mathtt{TwSum}(P_{\mathcal{J}})^P\rightarrow \Ch^P
\]
that carry any $P$-algebra in $\mathtt{TwSum}(P_{\mathcal{J}})$,
represented by a symmetric monoidal functor $\tilde{C}: \mathtt{cat}(P)\rightarrow \mathtt{TwSum}(P_{\mathcal{J}})$, to the $P$-algebra in $\Ch$
represented by the composite functors $\tilde{A}\tilde{C},\tilde{B}\tilde{C}: \mathtt{cat}(P)\rightarrow \Ch$.
We also have a natural transformation $\tilde{\phi}_*: \tilde{A}_*\Rightarrow\tilde{B}_*$ between these functors on $P$-algebras.

For any small category $\mathcal{I}$, we get strict symmetric monoidal dg functors
\[
\tilde{A}_*,\tilde{B}_*:(\mathtt{TwSum}(P_{\mathcal{J}})^P)^{\mathcal{I}}\rightarrow (\Ch^P)^{\mathcal{I}}
\]
and a strict symmetric monoidal dg natural transformation $\tilde{\phi}_*:\tilde{A}_*\Rightarrow\tilde{B}_*$.
This transformation consists in a collection of natural transformations of $\mathcal{I}$-diagrams of dg $P$-algebras
\[
\tilde{\phi}_*(Y):\tilde{A}_*(Y)\Rightarrow\tilde{B}_*(Y)
\]
for every $Y\in (\mathtt{TwSum}(P_{\mathcal{J}})^P)^{\mathcal{I}}$.

Thus, whenever we have an $\mathcal{I}$-diagram of $P$-algebras in $\mathtt{TwSum}(P_{\mathcal{J}})$, say $Y$, we can
associate an $\mathcal{I}$-diagram of dg $P$-algebras $\tilde{A}_*(Y)$ to any $\mathcal{J}$-diagram of dg $P$-algebras $A$,
and a natural transformation of $\mathcal{I}$-diagrams of dg $P$-algebras $\tilde{\phi}_*(Y): \tilde{A}_*(Y)\Rightarrow\tilde{B}_*(Y)$ to any natural transformation of $\mathcal{J}$-diagrams of dg $P$-algebras $\phi: A\Rightarrow B$.
This result is equivalent to the following statement:
\begin{prop}
Given an $\mathcal{I}$-diagram $Y$ of $P$-algebras in $\mathtt{TwSum}(P_{\mathcal{J}})$, the above construction determines a functor
\[
(\Ch^P)^{\mathcal{J}}\rightarrow (\Ch^P)^{\mathcal{I}}.
\]
\end{prop}

The main example to which we want to apply this construction is the following.
Let $f:X\rightarrow Y$ be a morphism of chain complexes, then it admits a functorial factorization
by an acyclic cofibration (i.e acyclic injection) followed by a fibration (i.e a surjection).
This factorization is explicitely given by
\[
\Xi(f:X\rightarrow Y):
\xymatrix{
 & & X \\
X \ar@{>->}[]!R+<4pt,0pt>;[r]_-i^-{\sim} \ar@/^/[urr]^{id_X}\ar@/_/[drr]_f & Z \ar@{->>}[ur]^-s \ar@{->>}[dr]_-p & \\
 & & Y
}
\]
where
\[
Z=(\mathbb{K}\mathbb{e}_0\otimes X\oplus \mathbb{K}\mathbb{e}_{01}\otimes Y\oplus \mathbb{K}\mathbb{e}_1\otimes Y,d_Z)
\]
with $deg(\mathbb{e}_0)=deg(\mathbb{e}_1)=0$ and $deg(\mathbb{e}_{01})=-1$.
The differential $d_Z$ can be expressed in this direct sum by the matrix
\[
\begin{pmatrix}
d_X & 0 & 0 \\
f & -d_Y & -id \\
0 & 0 & d_Y \\
\end{pmatrix}
=
\begin{pmatrix}
d_X & 0 & 0 \\
0 & -d_Y & 0 \\
0 & 0 & d_Y \\
\end{pmatrix}
+
\begin{pmatrix}
0 & 0 & 0 \\
f & 0 & -id \\
0 & 0 & 0 \\
\end{pmatrix}
\]
where the first matrix of the sum is the differential of the direct sum $\mathbb{K}\mathbb{e}_0\otimes X\oplus \mathbb{K}\mathbb{e}_{01}\otimes Y\oplus \mathbb{K}\mathbb{e}_1\otimes Y$
and the second is a twisting $\mathtt{tw}_Z$, a map of degree $-1$ satisfying $\mathtt{tw}_Z^2=0$.
The map $i$ sends $x\in X$ to $x\oplus 0\oplus f(x)$ and $s$ and $p$ are respectively projections on the first and the third factor, that is, we have
\[
i =
\begin{pmatrix}
id & 0 & f \\
\end{pmatrix},
\]
\[
s =
\begin{pmatrix}
id \\ 0 \\ 0 \\
\end{pmatrix}
\]
and
\[
p=
\begin{pmatrix}
0 \\ 0 \\ id \\
\end{pmatrix}.
\]
There is a diagram of chain complexes
\[
\Xi:Mor(\Ch)\rightarrow Fun(\mathcal{Y},\Ch)
\]
functorial in its variables,
where $Mor(\Ch)$ is the category whose objects are morphisms of chain complexes and morphisms are
commutative squares, and $\mathcal{Y}$ is the small category whose objects and arrows are given by
\begin{equation*}
\mathcal{Y} := \left\{\vcenter{
\xymatrix{
 & & \bullet \\
\bullet \ar[r] \ar@/^/[urr]\ar@/_/[drr] & \bullet \ar[ur] \ar[dr] & \\
 & & \bullet
}
}\right\}.
\end{equation*}
Our goal is to prove that for any cofibrant dg prop $P$, this functor induces a functor
\[
\Xi:Mor(\Ch^P)\rightarrow Fun(\mathcal{Y},\Ch^P),
\]
that is, a functor
\[
\Xi:(\Ch^P)^{\bullet\rightarrow\bullet}\rightarrow (\Ch^P)^{\mathcal{Y}}.
\]
This means the following:
\begin{thm}
Let $P$ be a cofibrant dg prop. The functorial factorization of morphism of chain complexes described above lifts to a functorial factorization of $P$-algebra morphisms into an acyclic injection followed by a surjection.
\end{thm}
\begin{proof}
The general strategy is to prove that the diagram in $\mathtt{TwSum}(P_{x\rightarrow y})$ associated to
$\Xi(f:X\rightarrow Y)$ is actually a diagram in $\mathtt{TwSum}(P_{x\rightarrow y})^P$ and then apply Proposition 2.8.

Let $f:X\rightarrow Y$ be a morphism of chain complexes and $P_{x\rightarrow y}$ be the $2$-colored dg prop of $P$-algebra morphisms.
In this proof, we will use the short notation
\[
\mathtt{Tw}:=\mathtt{TwSum}(P_{x\rightarrow y}).
\]
We can associate to the diagram of chain complexes $\Xi (f:X\rightarrow Y)$ a diagram $\Xi(f:x\rightarrow y)$ in $\mathtt{Tw}$ so that Proposition 2.8 applies.
For this, recall that the colors $x$ and $y$ are embedded into $\mathtt{Tw}$ as the objects $\mathbb{K}\mathbb{e}_0\otimes x,0)$
and $\mathbb{K}\mathbb{e}_1\otimes y,0)$. We will denote by $f$ both the operation of $P_{x\rightarrow y}$ corresponding to $f$ and the morphism
$\mathbb{K}\mathbb{e}_0\otimes x,0)\rightarrow \mathbb{K}\mathbb{e}_1\otimes y,0)$ in $\mathtt{Tw}$.
The object $z$ of $\mathtt{Tw}$ corresponding to $Z$ is defined by
\[
(\mathbb{K}\mathbb{e}_0\otimes x\oplus \mathbb{K}\mathbb{e}_{01}\otimes y\oplus \mathbb{K}\mathbb{e}_1\otimes y,\mathtt{tw}_z)
\]
with
\[
\mathtt{tw}_z =
\begin{pmatrix}
\mathtt{tw}_{0,0} & \mathtt{tw}_{01,0} & \mathtt{tw}_{1,0} \\
\mathtt{tw}_{0,01} & \mathtt{tw}_{01,01} & \mathtt{tw}_{1,01} \\
\mathtt{tw}_{0,1} & \mathtt{tw}_{01,1} & \mathtt{tw}_{1,1} \\
\end{pmatrix}
=
\begin{pmatrix}
0 & 0 & 0 \\
\mathbb{e}_{01}\otimes\mathbb{e}_0^{\vee}\otimes f & 0 & \mathbb{e}_{01}\otimes\mathbb{e}_1^{\vee}\otimes -id \\
0 & 0 & 0 \\
\end{pmatrix}
\]
representing the twisting part $\mathtt{tw}_Z$ of $Z$. The maps $i$ and $p$ of $\Xi(f:x\rightarrow y)$ are then defined similarly to those of $\Xi(f:X\rightarrow Y)$.

The endomorphism dg prop $End_{(\Xi(f:x\rightarrow y), \mathtt{Tw})}$ projects to the endomorphism dg prop $End_{(f, \mathtt{Tw})}$ of the subdiagram $f:x\rightarrow y$, hence we have a fibration of dg props
\[
End_{(\Xi(f:x\rightarrow y), \mathtt{Tw})}\twoheadrightarrow End_{(f, \mathtt{Tw})}.
\]
We will denote these dg props by $End_{\Xi(f:x\rightarrow y)}$ and $End_f$ for short.
We have to prove that this fibration is acyclic. For this, we consider the following commutative diagram of $\mathbb{S}$-biobjects:
\[
\xymatrix{
End_z=Hom_{zz} \ar[dr]^-{(i^*,p_*)}\ar@/^{1pc}/[rrd]^-{p_*}\ar@/^{-1pc}/[ddr]_-{i^*} & & \\
 & Hom_{xz}\times_{Hom_{xy}}Hom_{zy} \ar[d]\ar[r] & Hom_{zy} \ar[d]^{i^*}\\
 & Hom_{xz} \ar[r]_{p_*} & Hom_{xy}
}
\]
where $Hom_{zz}(m,n)= Hom_{\mathtt{Tw}}(z^{\otimes m},z^{\otimes n})$.
Limits of $\mathbb{S}$-biobjects are created pointwise, so for every $(m,n)\in\mathbb{N}^2$ we have a commutative diagram
\[
\xymatrix{
Hom_{\mathtt{Tw}}(z^{\otimes m},z^{\otimes n}) \ar[dr]^-{((i^{\otimes m})^*,(p^{\otimes n})_*)}\ar@/^{1pc}/[rrd]^-{(p^{\otimes n})_*}\ar@/^{-1pc}/[ddr]_-{(i^{\otimes m})^*} & & \\
 & pullback
\ar[d]\ar[r] & Hom_{\mathtt{Tw}}(z^{\otimes m},y^{\otimes n}) \ar[d]^{(i^{\otimes m})^*}\\
 & Hom_{\mathtt{Tw}}(x^{\otimes m},z^{\otimes n}) \ar[r]_{(p^{\otimes n})_*} & Hom_{\mathtt{Tw}}(x^{\otimes m},y^{\otimes n})
}.
\]
We have to check that $((i^{\otimes m})^*,(p^{\otimes n})_*)$ is an acyclic fibration.
Since acyclic fibrations of $\mathbb{S}$-biobjects are determined pointwise, we deduce that
\[
(i^*,p_*):End_z\stackrel{\sim}{\twoheadrightarrow} Hom_{xz}\times_{Hom_{xy}}Hom_{yz}
\]
is an acyclic fibration of $\Sigma$-objects.
Let us consider now the base extensions
\[
End_x\times_{Hom_{xz}}End_z\times_{Hom_{zy}}End_y = End_{\Xi(f:x\rightarrow y)}
\]
and
\[
End_x\times_{Hom_{xz}}(Hom_{xz}\times_{Hom_{xy}}Hom_{zy})\times_{Hom_{zy}}End_y = End_f.
\]
Acyclic fibrations are stable under base extensions, and acyclic fibrations of dg props are determined in the category of $\mathbb{S}$-biobjects under the forgetful functor,
so we finally get the desired acyclic fibration of dg props
\[
End_x\times_{Hom_{xz}}(i^*,p_*)\times_{Hom_{zy}}End_y: End_{\Xi(f:x\rightarrow y)} \stackrel{\sim}{\twoheadrightarrow} End_f.
\]

Now let us note $X_b=\mathbb{K}\mathbb{e}_0$, $Y_b=\mathbb{K}\mathbb{e}_1$ and $f_b:X_b\rightarrow Y_b$ the morphism sending $\mathbb{e}_0$
to $\mathbb{e}_1$. This morphism admits a factorization
\[
\xymatrix{
X_b \ar@{>->}[]!R+<4pt,0pt>;[r]_-{i_b}^-{\sim} & Z_b \ar@{->>}[r]_-{p_b} & Y_b.
}
\]
Our goal is to prove that for every natural integers $m$ and $n$, we have isomorphisms of chain complexes
\[
Hom_{\mathtt{Tw}}(z^{\otimes m},z^{\otimes n}) \cong Hom_{\Ch}(Z_b^{\otimes m},Z_b^{\otimes n})\otimes P(m,n),
\]
\[
Hom_{\mathtt{Tw}}(z^{\otimes m},y^{\otimes n}) \cong Hom_{\Ch}(Z_b^{\otimes m},Y_b^{\otimes n})\otimes P(m,n)
\]
and
\[
Hom_{\mathtt{Tw}}(x^{\otimes m},z^{\otimes n}) \cong Hom_{\Ch}(X_b^{\otimes m},Z_b^{\otimes n})\otimes P(m,n).
\]
The method is exactly the same for the three cases, so we just write the argument for the third isomorphism.
We need to determine the tensor powers of $z$. For every natural integer $n$, the object $z^{\otimes n}$ is given by the direct sum of shuffles
\[
\bigoplus_{\substack{1\leq j\leq i\leq n\\ \sigma\in Sh(i,m-i)\\ \tau\in Sh(j,m-j)}}\sigma_*((\mathbb{K}\mathbb{e}_0\otimes x)^{\otimes n-i},\tau_*((\mathbb{K}\mathbb{e}_{01}\otimes y)^{\otimes j},
(\mathbb{K}\mathbb{e}_1\otimes y)^{\otimes i-j})),
\]
where the action $\sigma_*(A^{\otimes k},B^{\otimes l})$ of a $(k,l)$-shuffle $\sigma$ on a pair of tensor powers $(A^{\otimes k},B^{\otimes l})$ permutes the variables of the tensor product
$A^{\otimes k}\otimes B^{\otimes l}$.
The twisting of $z^{\otimes n}$ is determined by
\[
\mathtt{tw}_{0,01}^{\otimes n}=\mathbb{e}_{01}^{\otimes n}\otimes (\mathbb{e}_0^{\vee})^{\otimes n}\otimes f^{\circ_h n}
\]
and
\[
\mathtt{tw}_{1,01}^{\otimes n}=\mathbb{e}_{01}^{\otimes n}\otimes (\mathbb{e}_1^{\vee})^{\otimes n}\otimes (-id)^{\circ_h n}
\]
where $\circ_h$ is the horizontal composition product of the dg prop $P_{x\rightarrow y}$.
We get
\begin{multline*}
Hom_{\mathtt{Tw}}(x^{\otimes m},z^{\otimes n}) \\
 = \\
\bigoplus_{\substack{1\leq j\leq i\leq n\\ \sigma\in Sh(i,m-i)\\ \tau\in Sh(j,m-j)}}
Hom_{\mathtt{Tw}}(x^{\otimes m},\sigma_*((\mathbb{K}\mathbb{e}_0\otimes x)^{\otimes n-i},\tau_*((\mathbb{K}\mathbb{e}_{01}\otimes y)^{\otimes j},(\mathbb{K}\mathbb{e}_1\otimes y)^{\otimes i-j})))\\
 \cong \\
\bigoplus_{\substack{1\leq j\leq i\leq n\\ \sigma\in Sh(i,m-i)\\ \tau\in Sh(j,m-j)}}\mathbb{K}\mathbb{e}_0^{\otimes n-i}\otimes\mathbb{K}\mathbb{e}_{01}^{\otimes j}\otimes\mathbb{K}\mathbb{e}_1^{\otimes i-j}\otimes Hom_{\mathtt{Tw}}(x^{\otimes m},
\sigma_*(x^{\otimes n-i},\tau_*(y^{\otimes j},y^{\otimes i-j})))\\
\end{multline*}
Moreover, we have
\[
Hom_{\mathtt{Tw}}(x^{\otimes m},\sigma_*(x^{\otimes n-i},\tau_*(y^{\otimes j},y^{\otimes i-j})))= P_{x\rightarrow y}(x,\cdots,x;\sigma_*(x,\cdots,x,\tau_*(y,\cdots,y)))
\]
where $P_{x\rightarrow y}(x,\cdots,x;\sigma_*(x,\cdots,x,\tau_*(y,\cdots,y)))$ has $m$ copies of the color $x$ in input, and in output $n-i$ colors $x$ and $i$ colors $y$
permuted by the shuffles $\sigma$ and $\tau$.
We want to build an isomorphism
\begin{multline*}
\bigoplus_{\substack{1\leq j\leq i\leq n\\ \sigma\in Sh(i,m-i)\\ \tau\in Sh(j,m-j)}} \mathbb{K}\mathbb{e}_0^{\otimes n-i}\otimes\mathbb{K}\mathbb{e}_{01}^{\otimes j}\otimes\mathbb{K}\mathbb{e}_1^{\otimes i-j}\otimes
P_{x\rightarrow y}(x,\cdots,x;\sigma_*(x,\cdots,x,\tau_*(y,\cdots,y))) \\
 \cong \\
\bigoplus_{\substack{1\leq j\leq i\leq n\\ \sigma\in Sh(i,m-i)\\ \tau\in Sh(j,m-j)}} Hom_{\Ch}(X_b^{\otimes m},\sigma_*(X_b^{\otimes n-i},\tau_*(Y_b[-1]^{\otimes j},Y_b^{\otimes i-j})))\otimes P(m,n)\\
\end{multline*}
where $[-1]$ is the degree shift applied to the chain complex $Y_b$.
For this, we define in each component $(i,j,\sigma,\tau)$ of the direct sum an isomorphism
\begin{eqnarray*}
\mathbb{K}\mathbb{e}_0^{\otimes n-i}\otimes\mathbb{K}\mathbb{e}_{01}^{\otimes j}\otimes\mathbb{K}\mathbb{e}_1^{\otimes i-j}\otimes
P_{x\rightarrow y}(x,\cdots,x;\sigma_*(x,\cdots,x,\tau_*(y,\cdots,y))) & & \\
\rightarrow & & \\
Hom_{\Ch}(X_b^{\otimes m},\sigma_*(X_b^{\otimes n-i},\tau_*(Y_b[-1]^{\otimes j},Y_b^{\otimes i-j})))\otimes P(m,n) & & \\
\end{eqnarray*}
which sends any
\[
\xi\in P_{x\rightarrow y}(x,\cdots,x;\sigma_*(x,\cdots,x,\tau_*(y,\cdots,y)))
\]
to
\[
\sigma_*\tau_*\otimes \sigma_*(f^{\circ_h n-i},id^{\circ_h i})\circ_v\xi,
\]
where $\sigma_*\tau_*$ is the unique homomorphism sending $\mathbb{e}_0^{\otimes m}$ to $\sigma_*(\mathbb{e}_0^{\otimes n-i},\tau_*(\mathbb{e}_{01}^{\otimes j},\mathbb{e}_1^{\otimes i-j}))$
and $\sigma_*(f^{\circ_h n-i},id^{\circ_h i})$ is the permutation of the variables in the iterated horizontal product $f\circ_h \cdots\circ_h f\circ_h id\circ_h \cdots\circ_h id$
by $\sigma$.

Finally, since $((i^{\otimes m})^*,(p^{\otimes n})_*)$ is the tensor product of $((i_b^{\otimes m})^*,(p_b^{\otimes n})_*)$ by $P(m,n)$,
it remains to apply the methods of \cite[Lemma 8.3]{Fre1} in the category of chain complexes, for $X_b$ and $Y_b$,
to prove that $((i_b^{\otimes m})^*,(p_b^{\otimes n})_*)$ is an acyclic fibration.
We write the arguments here for the sake of clarity. We have a commutative diagram
\[
\xymatrix{
Hom_{\Ch}(Z_b^{\otimes m},Z_b^{\otimes n}) \ar[dr]^-{((i_b^{\otimes m})^*,(p_b^{\otimes n})_*)}\ar@/^{1pc}/[rrd]^-{(p_b^{\otimes n})_*}\ar@/^{-1pc}/[ddr]_-{(i_b^{\otimes m})^*} & & \\
 & pullback
\ar[d]\ar[r] & Hom_{\Ch}(Z_b^{\otimes m},Y_b^{\otimes n}) \ar[d]^{(i_b^{\otimes m})^*}\\
 & Hom_{\Ch}(X_b^{\otimes m},Z_b^{\otimes n}) \ar[r]_{(p_b^{\otimes n})_*} & Hom_{\Ch}(X_b^{\otimes m},Y_b^{\otimes n})
}.
\]
Recall that chain complexes over a field are all cofibrant and fibrant in the model structure of $\Ch$.
The map $i_b$ is a cofibration and $X_b$ is cofibrant, so by the pushout-product axiom, for every integer $n$ the map $i_b^{\otimes n}:X^{\otimes n}\rightarrow Z^{\otimes n}$
is a cofibration.
The category $\Ch$ satisfies the limit monoid axioms \cite[Section 6]{Fre1} and $Y_b$ is fibrant, so for every integer $n$ the map $p_b^{\otimes n}:Z_b^{\otimes n}\rightarrow Y_b^{\otimes n}$ is a fibration \cite[Proposition 6.7]{Fre1}.
Moreover, by the pushout-product axiom, the tensor product preserves acyclic cofibrations between cofibrant objects, so by Brown's Lemma it preserves
weak equivalences between cofibrant objects. Given that $Z_b$ and $Y_b$ are cofibrant, it implies that $p_b^{\otimes n}$ is an acyclic fibration.
According to the dual pushout-product axiom, the fact that $i_b^{\otimes m}$ is a cofibration and $p_b^{\otimes n}$ is an acyclic fibration implies that
$((i_b^{\otimes m})^*,(p_b^{\otimes n})_*)$ is an acyclic fibration.
\end{proof}

\section{The subcategory of acyclic fibrations}

The goal of this section is to prove that the classifying space of weak equivalences of $P$-algebras
is weakly equivalent to the classifying space of acyclic fibrations of $P$-algebras:
\begin{thm}
Let $P$ be a cofibrant dg prop.
The inclusion of categories $i: fw\Ch^P\hookrightarrow w\Ch^P$
gives rise to a weak equivalence of simplicial sets $\mathcal{N} fw\Ch^{P}\stackrel{\sim}{\rightarrow}\mathcal{N} w\Ch^{P}$.
\end{thm}
\begin{rem}
Actually, the methods of \cite{Yal2} can be transposed in our setting to prove the following much stronger statement.
We refer the reader to the seminal papers \cite{DK1}, \cite{DK2} and \cite{DK3} for the notions of simplicial localization,
hammock localization and Dwyer-Kan equivalences of simplicial categories.
The inclusion of categories $i:fw\Ch^P\hookrightarrow w\Ch^P$ induces a Dwyer-Kan equivalence of hammock
localizations
\[
L^H(\Ch^P,fw\Ch^P)\stackrel{\sim}{\rightarrow} L^H(\Ch^P,w\Ch^P).
\]
We refer the reader to \cite{Yal2} for more details about this proof, which relies on the properties of several models of
$(\infty,1)$-categories (simplicial categories \cite{Ber}, relative categories \cite{BK} and complete Segal spaces \cite{Rez2}).
\end{rem}

To prove this theorem, we use Quillen's Theorem A \cite{Qui}: we have to check that for every chain complex $X$,
the nerve of the comma category $(X\downarrow i)$ is contractible. For this aim, we prove
the following more general result:
\begin{prop}
Let $\mathcal{I}$ be a small category. Every simplicial map $\mathcal{N}\mathcal{I}\rightarrow\mathcal{N}(X\downarrow i)$ is null up to homotopy.
\end{prop}
As a consequence we get:
\begin{prop}
The simplicial set $\mathcal{N}(X\downarrow i)$ is contractible.
\end{prop}
To prove Proposition 3.4, we apply Proposition 3.3, for every $n\in\mathbb{N}$, to the subdivision
category of a simplicial model of the $n$-sphere $S^n$.
We take $\partial\Delta^{n+1}$ as simplicial model of $S^n$ and note $sd\partial\Delta^{n+1}$ its subdivision category.
We then use general arguments of homotopical algebra:
\begin{prop}
Let $F:\mathcal{C}\rightleftarrows\mathcal{D}:G$ be a Quillen adjunction. It induces natural isomorphisms
\[
Map_{\mathcal{D}}(F(X),Y) \cong Map_{\mathcal{C}}(X,G(Y))
\]
where $X$ is a cofibrant object of $\mathcal{C}$ and $Y$ a fibrant object of $\mathcal{D}$.
\end{prop}
\begin{proof}
We will use the definition of mapping spaces via cosimplicial frames.
The proof works as well with simplicial frames.
The adjunction $(F,G)$ induces an adjunction at the level of diagram categories
\[
F:\mathcal{C}^{\Delta}\rightleftarrows\mathcal{D}^{\Delta}:G.
\]
Now let $\phi:A^{\bullet}\rightarrowtail B^{\bullet}$ be a Reedy cofibration between Reedy cofibrant objects
of $\mathcal{C}^{\Delta}$. This is equivalent, by definition, to say that for every integer $r$ the maps
\[
(\lambda,\phi)_r:L^rB\coprod_{L^rA}A^r\rightarrowtail B^r
\]
induced by $\phi$ and the latching object construction $L^{\bullet}A$ are cofibrations in $\mathcal{C}$.
Let us consider the pushout
\[
\xymatrix{
L^rA\ar[r]\ar[d]_{L^r\phi} & A^r\ar[d] \\
L^rB\ar[r] & L^rB\coprod_{L^rA}A^r
}.
\]
The fact that $\phi$ is a Reedy cofibration implies that for every $r$, the map $L^r\phi$ is a cofibration.
Since cofibrations are stable under pushouts, the map $A^r\rightarrow L^rB\coprod_{L^rA}A^r$ is also
a cofibration. By assumption, the cosimplicial object $A^{\bullet}$ is Reedy cofibrant, so it is in particular
pointwise cofibrant. We deduce that $L^rB\coprod_{L^rA}A^r$ is cofibrant. Similarly, each $B^r$ is cofibrant
since $B^{\bullet}$ is Reedy cofibrant. The map $(\lambda,\phi)_r$ is a cofibration between cofibrant objects
and $F$ is a left Quillen functor, so $F((\lambda,\phi)_r)$ is a cofibration of $\mathcal{D}$ between
cofibrant objects. Recall that the $r^{th}$ latching object construction is defined by a colimit.
As a left adjoint, the functor $F$ commutes with colimits so we get a cofibration
\[
L^rF(B^{\bullet})\coprod_{L^rF(A^{\bullet})}F(A^r)\rightarrowtail F(B^r).
\]
This means that $F(\phi)$ is a Reedy cofibration in $\mathcal{D}^{\Delta}$.
Now, given that Reedy weak equivalences are the pointwise equivalences, if $\phi$ is a Reedy weak equivalence
between Reedy cofibrant objects then it is in particular a pointwise weak equivalence between pointwise cofibrant
objects, hence $F(\phi)$ is a Reedy weak equivalence in $\mathcal{D}^{\Delta}$.
We conclude that $F$ induces a left Quillen functor between cosimplicial objects for the Reedy model structures.
In particular, it sends any cosimplicial frame of a cofibrant object $X$ of $\mathcal{C}$ to a cosimplicial
frame of $F(X)$.
\end{proof}
\begin{rem}
The isomorphism above holds if the cosimplicial frame for the left-hand mapping space is chosen to be
the image under $F$ of the cosimplicial frame of the right-hand mapping space.
But recall that cosimplicial frames on a given object are all weakly equivalent, so that for any choice
of cosimplicial frame we get at least weakly equivalent mapping spaces.
\end{rem}
Now, recall that the geometric realization functor and the singular complex functor induce a Quillen equivalence
\[
|-|:sSet\rightleftarrows Top:Sing_{\bullet}(-)
\]
between topological spaces and simplicial sets.
We have
\begin{eqnarray*}
Map_{sSet}(\mathcal{N}sd\partial\Delta^{n+1},\mathcal{N}(X\downarrow i)) & \simeq & Map_{sSet}(\mathcal{N}sd\partial\Delta^{n+1},Sing_{\bullet}(|\mathcal{N}(X\downarrow i)|)) \\
 & \simeq & Map_{Top}(|\mathcal{N}sd\partial\Delta^{n+1}|,|\mathcal{N}(X\downarrow i)|) \\
 & \simeq & Map_{Top}(S^n,|\mathcal{N}(X\downarrow i)|) \\
\end{eqnarray*}
hence
\[
|Map_{sSet}(\mathcal{N}sd\partial\Delta^{n+1},\mathcal{N}(X\downarrow i))|\simeq |Map_{Top}(S^n,|\mathcal{N}(X\downarrow i)|)|,
\]
in particular
\[
\pi_0|Map_{sSet}(\mathcal{N}sd\partial\Delta^{n+1},\mathcal{N}(X\downarrow i))| \cong [S^n,|\mathcal{N}(X\downarrow i)|]_{Ho(Top)}.
\]
Proposition 3.3 means that for every integer $n$, the space $|Map_{sSet}(\mathcal{N}sd\partial\Delta^{n+1},\mathcal{N}(X\downarrow i))|$
has only one connected component (the component of the zero map),
that is, the homotopy groups of $|\mathcal{N}(X\downarrow i)|$ are trivial.

\begin{proof}[Proof of Proposition 3.3]
The category $(X\downarrow i)$ has weak equivalences $X\stackrel{\sim}{\rightarrow} Y$ as objects and acyclic fibrations as morphisms.
It contains the initial object $X\stackrel{=}{\rightarrow}X$ of $(X\downarrow \Ch)$.

Every simplicial map $\mathcal{N}\mathcal{I}\rightarrow \mathcal{N}(X\downarrow i)$ comes from a functor
$\mathcal{I}\rightarrow (X\downarrow i)$, i.e a $\mathcal{I}$-diagram in $(X\downarrow i)$.
Let $F$ be such a functor.
Let $\overline{X}$ be the initial $\mathcal{I}$-diagram, that is the constant diagram on $X\stackrel{=}{\rightarrow} X$.
In order to simplify notations, we note $Y$ for a morphism $X\rightarrow Y$ (an object of $(X\downarrow \Ch)$)
and $Y\rightarrow Y'$ for a commutative triangle relating $X\rightarrow Y$ to $X\rightarrow Y'$
(a morphism of $(X\downarrow \Ch)$).
The diagram $F\times\overline{X}:\mathcal{I}\rightarrow (X\downarrow \Ch)$ is defined on objects by
$F\times\overline{X}(k)=F(k)\times X$ an on arrows by $F\times\overline{X}(\phi)=F(\phi)\times id_X$.
Applying the functorial factorization of Theorem 2.9
to the unique initial morphism $\overline{X}\rightarrow F\times\overline{X}$, we get a decomposition in
$(X\downarrow \Ch)^P$ into a diagram $\mathcal{Y}$ given by
\[
\xymatrix{ &  & \overline{X}\\
\overline{X}\ar@/^{1pc}/[urr]^{=}\ar@/^{-1pc}/[drr]\ar@{>->}[]!R+<4pt,0pt>;[r]_{i}^-{\sim} & G\ar@{->>}[ur]_{p_{1}}\ar@{->>}[dr]^{p_{2}}\\
 &  & F
}.
\]
where the functor $G$ is defined pointwise by the functorial factorization of Theorem 2.9. The map $(p_{1},p_{2}):G\twoheadrightarrow F\times\overline{X}$ is a pointwise fibration and $i$ is a pointwise acyclic cofibration
of chain complexes.
Since the map $(p_{1},p_{2}):G\twoheadrightarrow F\times\overline{X}$ is a pointwise fibration
and $F$ and $\overline{X}$ are pointwise fibrant,
the maps $p_1$ and $p_2$ are pointwise acyclic fibrations: the product $F\times\overline{X}$ is given by the pullback
\[
\xymatrix{
F\times\overline{X} \ar[r]^{p_1} \ar[d]_{p_2} & X \ar[d] \\
F \ar[r] & \bullet }
\]
and pointwise fibrations are stable under pullbacks so $p_1$ and $p_2$ are pointwise fibrations.
Since $id_{\overline{X}}=p_1\circ i$ and $\overline{X}\rightarrow F=p_2\circ i$ are weak equivalences,
the maps $p_1$ and $p_2$ are acyclic by the two-out-of-three property.

The functors $\overline{X}$ and $F$ take their values in $(X\downarrow i)$ by definition.
This implies that the functor $G$ sends morphisms of $\mathcal{I}$ to acyclic fibrations
by definition of the functorial factorization in chain complexes.
We obtain consequently a zigzag of natural transformations $\overline{X}\leftarrow G \rightarrow F$ of functors
$\mathcal{I}\rightarrow (X\downarrow i)$. This zigzag implies that $\mathcal{N}F$ is homotopic to
$\mathcal{N}\overline{X}$, which is itself null up to homotopy.
This concludes the proof of Proposition 3.3.
\end{proof}

\section{Moduli spaces of algebraic structures as homotopy fibers}

\subsection{Moduli spaces of algebra structures on fibrations}

The results of this subsection holds for algebras in $\mathcal{E}$ over a prop in $\mathcal{C}$, where
the category $\mathcal{C}$ is a cofibrantly generated symmetric monoidal model category and
the category $\mathcal{E}$ is a cofibrantly generated symmetric monoidal model category over $\mathcal{C}$.
However, for the sake of simplicity we explain only the case $\mathcal{E}=\mathcal{C}=\Ch$.

We start by recalling \cite[Lemma 7.2]{Fre1}. Let $f:A\rightarrow B$
be a morphism of $\Ch$, we have a pullback
\[
\xymatrix{End_{\{A\rightarrow^{f}B\}}\ar[r]^{d_{0}}\ar[d]_{d_{1}} & End_{B}\ar[d]^{f^{*}}\\
End_{A}\ar[r]_{f_{*}} & Hom_{AB}
}
\]
where $Hom_{AB}$ is defined by $Hom_{AB}(m,n)=Hom_{\Ch}(A^{\otimes m},B^{\otimes n})$.

\begin{lem}[{\cite[Lemma 7.2]{Fre1}}]

(1) If $f$ is a (acyclic) fibration then so is $d_{0}$.

(2) If f is a cofibration, then $d_{1}$ is a fibration. If $f$ is
also acyclic then $d_{1}$ is an acyclic fibration and $d_{0}$ a
weak equivalence.
\end{lem}

\begin{rem}
It is a generalization in the prop context of \cite[Proposition 4.1.7]{Rez} and \cite[Proposition 4.1.8]{Rez}.
\end{rem}

\begin{lem}
Let $X_{n}\twoheadrightarrow\cdots\twoheadrightarrow X_{1}\twoheadrightarrow X_{0}$
be a chain of fibrations of chain complexes. For every $0\leq k\leq n-1$, the map $d_{0}$ in the
pullback

\[
\xymatrix{End_{\{X_{n}\twoheadrightarrow\cdots\twoheadrightarrow X_{0}\}}\ar[r]^{d_{0}}\ar[d]_{d_{1}} & End_{\{X_{k}\twoheadrightarrow\cdots\twoheadrightarrow X_{0}\}}\ar[d]\\
End_{\{X_{n}\twoheadrightarrow\cdots\twoheadrightarrow X_{k+1}\}}\ar[r] & Hom_{X_{k+1}X_{k}}\ensuremath{}
}
\]
is a fibration . Moreover, if the fibrations in the chain
$X_{n}\twoheadrightarrow\cdots\twoheadrightarrow X_{1}\twoheadrightarrow X_{0}$
are acyclic then so is $d_{0}$.
\end{lem}
\begin{proof}
We prove this lemma by induction. The case $n=1$ is Lemma 4.1. Now
suppose that our lemma is true for a given integer $n\geq1$. Let
$X_{n+1}\twoheadrightarrow\cdots\twoheadrightarrow X_{1}\twoheadrightarrow X_{0}$
be a chain of fibrations of complexes. We distinguish two
cases:

\textbf{-the case $k=n$}: we have the pullback
\[
\xymatrix{End_{\{X_{n+1}\twoheadrightarrow\cdots\twoheadrightarrow X_{0}\}}\ar[r]^{d_{0}}\ar[d]_{d_{1}} & End_{\{X_{n}\twoheadrightarrow\cdots\twoheadrightarrow X_{0}\}}\ar[d]\\
End_{X_{n+1}}\ar[r]_{f_{*}} & Hom_{X_{n+1}X_{n}}\ensuremath{}
}
\]
where $f:X_{n+1}\twoheadrightarrow X_{n}$. The fact that $f$ is
a fibration implies that $f_{*}$ is a fibration, so $d_{0}$ is a
fibration because of the stability of fibrations under pullback, and
the acyclicity of $f$ implies the acyclicity of $d_{0}$. The detailed
proof of these statements is done in the proof of \cite[Lemma 7.2]{Fre1}.

\textbf{-the case $0\leq k\leq n-1$}: $d_{0}=End_{\{X_{n+1}\twoheadrightarrow\cdots\twoheadrightarrow X_{0}\}}\rightarrow End_{\{X_{n}\twoheadrightarrow\cdots\twoheadrightarrow X_{0}\}}\rightarrow End_{\{X_{k}\twoheadrightarrow\cdots\twoheadrightarrow X_{0}\}}$
is the composite of an map satisfying the induction hypothesis with
the map of the case $k=n$, so the conclusion follows.
\end{proof}

\begin{rem}
This lemma is the generalization in the prop context of \cite[Proposition 4.1.9]{Rez}.
\end{rem}
We deduce from lemmata 4.1 and 4.3 the following properties of our
moduli spaces:
\begin{prop}
Let $f:X\rightarrow Y$ be a chain complex morphism and $P$ be
a cofibrant dg prop. The pullback of lemma 5.9 gives
rise to the following diagram of simplicial sets:

\[
P\{X\}\stackrel{(d_{1})_{*}}{\leftarrow}P\{f\}\stackrel{(d_{0})_{*}}{\rightarrow}P\{Y\}
\]

(1) If $f$ is a cofibration then $(d_{1})_{*}$ is a fibration. Moreover,
if $f$ is acyclic then $(d_{0})_{*}$ and $(d_{1})_{*}$ are weak
equivalences.

(2) If $f$ is a fibration then $(d_{0})_{*}$ is a fibration. Moreover,
if $f$ is acyclic then $(d_{0})_{*}$ and $(d_{1})_{*}$ are weak
equivalences.
\end{prop}
\begin{proof}
(1) If $f$ is a cofibration then $d_{1}$ is a fibration. So $(d_{1})_{*}$
is a fibration of simplicial sets according to Proposition 1.11. If
$f$ is acyclic, then $d_{0}$ and $d_{1}$ are weak equivalences.
Every chain complex is fibrant and cofibrant, and fibrations of props are determined componentwise, so $End_{X}$ and $End_{Y}$
are fibrant props. This implies that $End_{\{f\}}$ is also fibrant.
We deduce from this and Proposition 1.11 that $(d_{0})_{*}$ and $(d_{1})_{*}$
are weak equivalences.

(2) The proof is the same as in the previous case.
\end{proof}
By induction we can also prove the following proposition:
\begin{prop}
Let $X_{n}\stackrel{\sim}{\twoheadrightarrow}\cdots\stackrel{\sim}{\twoheadrightarrow}X_{1}\stackrel{\sim}{\twoheadrightarrow}X_{0}$
be a chain of acyclic fibrations and $P$ be a cofibrant dg prop. For every $0\leq k\leq n-1$, the
map $(d_{0})_{*}$ is an acyclic fibration and $(d_{1})_{*}$ a weak
equivalence in the diagram below:

\[
P\{X_{n}\stackrel{\sim}{\twoheadrightarrow}\cdots\stackrel{\sim}{\twoheadrightarrow}X_{k+1}\}\stackrel{(d_{1})_{*}}{\leftarrow}P\{X_{n}\stackrel{\sim}{\twoheadrightarrow}\cdots\stackrel{\sim}{\twoheadrightarrow}X_{0}\}\stackrel{(d_{0})_{*}}{\rightarrow}P\{X_{k}\stackrel{\sim}{\twoheadrightarrow}\cdots\stackrel{\sim}{\twoheadrightarrow}X_{1}\}.
\]

\end{prop}

\begin{rem}
Propositions 4.5 and 4.6 are generalizations in the prop context
of \cite[Proposition 4.1.11]{Rez}, \cite[Proposition 4.1.12]{Rez} and \cite[Proposition 4.1.13]{Rez}.
\end{rem}

\subsection{Proof of Theorem 0.1}

We have now all the key results to generalize Rezk's theorem to algebras
over dg props. The remaining arguments are the same as that of Rezk,
so we will not repeat it with all details but essentially show how
our Theorem 3.1, as well as the main theorem of \cite{Yal}, fit in the proof.

Let $P$ be a cofibrant dg prop, and $\mathcal{N} w\Ch^{P\otimes\Delta^{\bullet}}$
the bisimplicial set defined by
\[
(\mathcal{N} w\Ch^{P\otimes\Delta^{\bullet}})_{m,n}=(\mathcal{N} w\Ch^{cf})^{P\otimes\Delta^n})_{m}.
\]
The dg prop $P$ is cofibrant, thus so is $P\otimes\Delta^n$ for every
$n\geq0$. According to Theorem 3.1, we have a weak equivalence
induced by an inclusion of categories
\[
\mathcal{N} fw\Ch^{P\otimes\Delta^n}\stackrel{\sim}{\rightarrow}\mathcal{N} w\Ch^{P\otimes\Delta^n}
\]
Morevoer, for every $n,n'\geq0$, $\Delta^n\rightarrow\Delta^{n'}$
induces a weak equivalence of cofibrant dg props $P\otimes\Delta^n\rightarrow P\otimes\Delta^{n'}$
and thereby a weak equivalence of simplicial sets
\[
\mathcal{N} w\Ch^{P\otimes\Delta^{n'}}\stackrel{\sim}{\rightarrow}\mathcal{N} w\Ch^{P\otimes\Delta^n}
\]
according to \cite[Theorem 0.1]{Yal}. We obtain a zigzag of weak equivalences
\[
diag\mathcal{N} fw\Ch^{P\otimes\Delta^{\bullet}}\stackrel{\sim}{\rightarrow}diag\mathcal{N} w\Ch^{P\otimes\Delta^{\bullet}}\stackrel{\sim}{\leftarrow}\mathcal{N} w\Ch^{P}
\]
We use an adaptation of a slightly modified version of Quillen's theorem
B (cf. \cite{Qui}), namely \cite[Lemma 4.2.2]{Rez}, in order
to determine the homotopy fiber of the map $diag\mathcal{N} fw\Ch^{P\otimes\Delta^{\bullet}}\rightarrow\mathcal{N} fw\Ch$.
To prove that our map verifies the hypotheses of this lemma we use
the propositions of Section 4.1 exactly in the same way as Rezk in
the operadic case. Then we check that $diag(U\downarrow X)\simeq P\{X\}$
where $U: fw\Ch^{P\otimes\Delta^{\bullet}}\rightarrow fw\Ch$
is the forgetful functor (by using again the propositions of Section
4.1) and finally we get the following diagram:
\[
\xymatrix{P\{X\}\ar[r]\ar[d] & diag\mathcal{N} fw\Ch^{P\otimes\Delta^{\bullet}}\ar[d]\ar[r]^{\sim} & diag\mathcal{N} w\Ch^{P\otimes\Delta^{\bullet}} & \mathcal{N} w\Ch^{P}\ar[l]_{\sim}\ar[d]\\
pt\ar[r] & \mathcal{N}fw\Ch\ar[rr]^{\sim} &  & \mathcal{N}w\Ch
}.
\]
The proof of Theorem 0.1 is complete.

\begin{rem}
Note that we can recover the transfer theorem of bialgebras structures of \cite[Theorem A]{Fre1} as a consequence
of Theorem 0.1. Indeed, let $P$ be a cofibrant dg prop in $\Ch$.
Let $X\stackrel{\sim}{\rightarrow}Y$ be a morphism of $\Ch$
such that $Y$ is endowed with a $P$-algebra structure. We have a
homotopy pullback of simplicial sets
\[
\xymatrix{P\{X\}\ar[d]\ar[r]^-{p} & \mathcal{N}w\Ch^{P}\ar[d]^{\mathcal{N}U}\\
\{X\}\ar[r] & \mathcal{N}w\Ch
}
\]
which induces an exact sequence of pointed sets
\[
\pi_{0}P\{X\}\rightarrow\pi_{0}\mathcal{N}w\Ch^{P}\rightarrow\pi_{0}\mathcal{N}w\Ch.
\]
The base point of the set $\pi_{0}\mathcal{N}w\Ch$
is the weak equivalence class of $X$, denoted by $[X]$. The weak
equivalence $X\stackrel{\sim}{\rightarrow}Y$ in $\Ch$
implies that we have the equality $[Y]=[X]$ and thus $\pi_{0}\mathcal{N}U([Y]_{P})=[X]$,
where $[Y]_{P}$ is the weak equivalence class of $Y$ in $\Ch^{P}$.
The exactness of the above sequence implies that $\pi_{0}p(P\{X\})=(\pi_{0}\mathcal{N}U)^{-1}([X])$
so $[Y]_{P}\in\pi_{0}p(P\{X\})$. This means that there exists a $P$-algebra
structure on $X$ such that we have a zigzag of $P$-algebras morphisms
\[
X\stackrel{\sim}{\leftarrow}\cdots\stackrel{\sim}{\rightarrow}Y
\]
which are weak equivalences of $\Ch$.
\end{rem}
\begin{rem}
We do not adress the case of simplicial sets. However \cite[Theorem 1.4]{JY} endows the algebras over a colored prop in simplicial
sets with a model category structure. Moreover, the free algebra functor
exists in this case. Therefore one can transpose the methods used
in the operadic setting to obtain a simplicial version of \cite[Theorem 0.1]{Yal}. Theorem 0.1 in simplicial sets can be proved by following step
by step Rezk's original proof. We also conjecture that our results
have a version in simplicial modules which would follow from arguments
similar to ours.
\end{rem}

\end{document}